\documentclass[12pt,a4paper]{article}
\usepackage{imsart}

\usepackage{fancyhdr}
\usepackage{titlesec}
\usepackage{titletoc}
\usepackage{graphicx}
\usepackage{amsmath}
\usepackage{amstext}
\usepackage{amsbsy}
\usepackage{amssymb}
\usepackage{theorem}
\usepackage{dsfont}
\usepackage{times}
\usepackage[T1]{fontenc}
\usepackage{eucal}
\usepackage{mathrsfs}
\usepackage{url}
\newcommand\dsone{\mathds{1}} 
\newcommand{\ds}{\displaystyle}
\newcommand{\ts}{\textstyle}
\newcommand{\tint}{{\ts \int}}

\newcommand{\C}[1]{\mathcal{#1}}

\newcommand{\ov}[1]{\overline{#1}}
\newcommand{\wt}[1]{\widetilde{#1}}
\newcommand{\wh}[1]{\widehat{#1}}
\newcommand{\B}[1]{\mathbb{#1}}

\DeclareMathOperator*{\ess}{ess}

\DeclareMathOperator{\Var}{\mathbb{V}ar}

\DeclareMathOperator{\Span}{span}

\numberwithin{equation}{section}
\newtheorem{thm}{Theorem}[section]
\newtheorem{prop}[thm]{Proposition}

\newtheorem{lemma}[thm]{Lemma}

\newtheorem{corollary}[thm]{Corollary}
{\theorembodyfont{\rm} 

\newtheorem{rmk}{Remark}[section]}
\theoremstyle{plain}

\theoremheaderfont{\sc}
\newenvironment{proof}{{\sc Proof.}}{\ $\square$}
\titleformat{\section}[block]{\sc\center}{\thetitle.}{5pt}{}[]
\titlespacing{\section}{0pt}{*4.5}{*3}
\titleformat{\subsection}[runin]{\sc}{\thetitle.}{5pt}{}[.]
\titlespacing{\subsection}{0pt}{*3}{*2}
\titleformat{\subsubsection}[runin]{\it}{\thetitle.}{5pt}{}[.]
\titlespacing{\subsubsection}{0pt}{*2}{*2}
\contentsmargin[30pt]{30pt}
\titlecontents{section}[15pt]{\vspace{6pt}\sc}
{\contentslabel[\thecontentslabel.]{15pt}}{\hspace*{-15pt}}
{\titlerule*[1pc]{.}\contentspage[\rm\thecontentspage]}
\titlecontents{subsection}[37pt]{\vspace{3pt}\small\sc}
{\contentslabel[\thecontentslabel.]{22pt}}{\hspace*{-22pt}}
{\titlerule*[1pc]{.}\contentspage[\normalsize\rm\thecontentspage]}
\titlecontents{subsubsection}[69pt]{\it}
{\contentslabel[\thecontentslabel.]{32pt}}{\hspace*{-32pt}}
{~\titlerule*[1pc]{.}\contentspage[\rm\thecontentspage]}

\newcommand{\thmref}[1]{\ref{#1} (page \pageref{#1})}
\newcommand{\myeq}[1]{{\rm (\ref{#1}, page \pageref{#1})}}

\allowdisplaybreaks

\DeclareMathOperator{\XX}{\text{\bf \textsf X}}
\DeclareMathOperator{\YY}{\text{\bf \textsf Y}}
\begin{document}
\renewcommand{\sectionmark}[1]{\markboth{\thesection\ #1}{}}
\renewcommand{\subsectionmark}[1]{\markright{\thesubsection\ #1}}
\fancyhf{}
\fancyhead[RE]{\small\sc\nouppercase{\leftmark}}
\fancyhead[LO]{\small\sc\nouppercase{\rightmark}}
\fancyhead[LE,RO]{\thepage}
\fancyfoot[RO,LE]{\small\sc Olivier Catoni $\Rightarrow$ Jean-Yves Audibert}
\fancyfoot[LO,RE]{\small\sc\today}
\renewcommand{\footruleskip}{1pt}
\renewcommand{\footrulewidth}{0.4pt}
\newcommand{\mypoint}{\makebox[1ex][r]{.\:\hspace*{1ex}}}
\addtolength{\footskip}{11pt}
\pagestyle{plain}
\begin{center}
{\bf Linear regression through PAC-Bayesian truncation}\\[12pt]
{\sc Jean-Yves Audibert\footnote{Universit\'e Paris-Est, Ecole des Ponts ParisTech, Imagine,
6 avenue Blaise Pascal, 77455 Marne-la-Vall\'ee, France, audibert@imagine.enpc.fr
}$^{\!,}$\footnote{Sierra, CNRS/ENS/INRIA --- UMR 8548, 45 rue d'Ulm, 75230 Paris cedex 05, France}, Olivier Catoni
\footnote{D\'epartement de Math\'ematiques et Applications, CNRS -- UMR 8553, 
\'Ecole Normale Sup\'erieure,
45 rue d'Ulm, 75230 Paris cedex 05, France, olivier.catoni@ens.fr}}$^{\!,}$\footnote{
INRIA Paris-Rocquencourt - CLASSIC team.}\\[12pt]
{\small \it \today }\\[12pt]
\end{center}

{\small
{\sc Abstract :} 
We consider the problem of predicting as well as the best linear combination of $d$ given functions in least squares regression under $L^\infty$ constraints on the linear combination. 
When the input distribution is known, there already exists an algorithm having an expected excess risk of order $d/n$, where $n$ is the size of the training data. 
Without this strong assumption, standard results often contain a multiplicative $\log n$ factor, complex constants involving the conditioning of the Gram matrix of the covariates, kurtosis coefficients or some
geometric quantity characterizing the relation between $L^2$ and $L^\infty$-balls and require
some additional assumptions like exponential moments of the output.

This work provides a PAC-Bayesian shrinkage procedure with a simple excess risk bound of order $d/n$ holding in expectation and in deviations, under various assumptions. 
The common surprising factor of these results is their simplicity and the absence of exponential moment condition on the output distribution while achieving exponential deviations.
The risk bounds are obtained through a PAC-Bayesian analysis on truncated differences of losses. 
We also show that these results can be generalized to other strongly convex loss functions. 
\\[12pt]
{\sc 2000 Mathematics Subject Classification:}
62J05, 62J07.\\[12pt]
{\sc Keywords:} 
Linear regression, Generalization error, Shrinkage, PAC-Bayesian
theorems, Risk bounds, Robust statistics, Resistant estimators,
Gibbs posterior distributions,  
Randomized estimators, Statistical learning theory
}

\tableofcontents

\newpage

\def\bmul{\begin{multline*}}
\def\bigbegar{\begin{eqnarray*}} 
\def\bigendar{\end{eqnarray*}}
\def\lbegar{$$\left\{ \begin{array}{lll}}
\def\rendar{\end{array} \right.$$}
\def\rendarp{\end{array} \right..$$}
\def\begarlab{\begin{equation} \begin{array}{lll} \label}
\def\endarlab{\end{array} \end{equation}}
\def\beglab{\begin{equation} \label}
\def\endlab{\end{equation}}
\newcommand\beglabc[1]{\begin{equation*} }
\def\endlabc{\end{equation*}}
\def\lbegarlab{\begin{equation} \left\{ \begin{array}{lll} \label}
\def\rendarlab{\end{array} \right. \end{equation}}
\def\rendarplab{\end{array} \right.. \end{equation}}

\newcommand\und[2]{\underset{#2}{#1}\;}
\newcommand\undc[2]{{#1}_{#2}\;}

\newcommand\wrt{\text{w.r.t.}} 

\newcommand\cst{\text{Cst}\,}
\newcommand\dsV{{\mathbb V}} 
\newcommand\eqdef{\triangleq}

\newcommand\A{\mathcal{A}}
\newcommand\cB{\mathcal{B}}
\newcommand\cC{\Theta}
\newcommand\cD{\mathcal{D}}
\newcommand\E{\mathbb{E}}
\newcommand\cE{\mathcal{E}}
\newcommand\cF{\mathcal{F}}
\newcommand\cH{\mathcal{H}}
\newcommand\I{\mathcal{I}}
\newcommand\J{\mathcal{J}}
\newcommand\cK{\mathcal{K}}
\newcommand\cL{\mathcal{L}}
\newcommand\M{\mathcal{M}}
\newcommand\N{\mathcal{N}}
\renewcommand\P{\mathbb{P}}
\newcommand\R{\mathbb{R}}
\newcommand\calR{\mathcal{R}}
\renewcommand\S{\mathcal{S}}
\newcommand\cS{\mathcal{S}}
\newcommand\bcR{B} 
\newcommand\cR{\tilde{B}} 
\newcommand\W{\mathcal{W}}
\newcommand\X{\mathcal{X}}
\newcommand\Y{\mathcal{Y}}
\newcommand\Z{\mathcal{Z}}

\newcommand\jyem{\em}

\newcommand\hdelta{\hat{\delta}}
\newcommand\hth{\hat{\theta}}
\newcommand\hbth{\bar{\th}}
\newcommand\tth{\tilde{\theta}}
\newcommand\hlam{\hat{\lam}}
\newcommand\lamerm{\hlam^{\textnormal{(erm)}}}

\newcommand\tb{\tilde{b}}
\newcommand\tc{\tilde{c}}
\newcommand\tf{\tilde{f}}
\newcommand\tlam{\tilde{\lam}}

\newcommand\be{\beta}
\newcommand\ga{\gamma}
\newcommand\kap{\kappa}
\newcommand\lam{\lambda}

\newcommand\sint{{\textstyle \int}}
\newcommand\inth{\sint \pi(d\theta)}
\newcommand\inthj{\sint \pij(d\theta)}
\newcommand\inthp{\sint \pi(d\theta')}

\newcommand\gas{\ga^*}

\newcommand\hpi{\hat{\pi}}
\newcommand\pis{\pi^*}
\newcommand\tpi{\tilde{\pi}}

\newcommand\ovth{\ov{\theta}}
\newcommand\wth{\wt{\theta}}
\newcommand\wthj{\wt{\theta}_j}

\newcommand\eps{\varepsilon}
\renewcommand\epsilon{\varepsilon}
\newcommand\logeps{\log(\eps^{-1})}
\newcommand\leps{\log^2(\eps^{-1})}

\newcommand{\bi}{\begin{itemize}}
\newcommand{\ei}{\end{itemize}}

\newcommand\ovR{\ov{R}}

\newcommand\hhpi{\hat{\hat{\pi}}}
\newcommand\wwth{\wt{\wt{\theta}}}
\newcommand\pia{\pi^{(1)}}
\newcommand\pib{\pi^{(2)}}
\newcommand\pij{\pi^{(j)}}
\newcommand\tpia{\tilde{\pi}^{(1)}}
\newcommand\tpib{\tilde{\pi}^{(2)}}
\newcommand\tpij{\tilde{\pi}^{(j)}}
\newcommand\wtha{\wt{\theta}_1}
\newcommand\wthb{\wt{\theta}_2}
\newcommand\wths{\wt{\theta}_s}
\newcommand\wtht{\wt{\theta}_t}

\newcommand\cmin{c_{\min}}
\newcommand\cmax{c_{\max}}

\newcommand\ra{\rightarrow}

\newcommand\hatt{\hat{t}}
\newcommand\argmax{\textnormal{argmax}}
\newcommand\argmin{\textnormal{argmin}}

\newcommand\diag{\textnormal{Diag}}

\newcommand\Fg{\cF^\#}

\newcommand\vp{\varphi}
\newcommand\tvp{\tilde{\vp}}
\newcommand\tphi{\tilde{\phi}}
\renewcommand\th{\theta}

\newcommand\vsp{\vspace{1cm}}
\newcommand\lhs{\text{l.h.s.}}
\newcommand\rhs{\text{r.h.s.}}

\newcommand\expe[2]{\undc{\E}{#1\sim#2}}
\newcommand\expec[2]{\undc{\E}{#1\sim#2}}
\newcommand\expecc[2]{\E_{#1}}
\newcommand\expecd[2]{\E_{#2}}

\newcommand\hC{\hat{C}}
\newcommand\hf{\hat{f}}
\newcommand\hrho{\hat{\rho}}

\newcommand\br{\bar{r}}
\newcommand\chr{\check{r}}
\newcommand\bR{\bar{R}}

\newcommand\lan{\langle}
\newcommand\ran{\rangle}

\newcommand\logepsg{\log(|\Cg|\eps^{-1})}

\newcommand\Pemp{\hat{\P}}

\newcommand\fracl[2]{{(#1)}/{#2}}
\newcommand\fracc[2]{{#1}/{#2}}
\newcommand\fracr[2]{{#1}/{(#2)}}
\newcommand\fracb[2]{{(#1)}/{(#2)}}

\newcommand\hfproj{\hf^{\textnormal{(proj)}}}
\newcommand\thproj{\hat{\th}^{\textnormal{(proj)}}}

\newcommand\hfols{\hf^{\textnormal{(ols)}}}
\newcommand\thfols{\tilde{f}^{\textnormal{(ols)}}}
\newcommand\thols{\hat{\th}^{\textnormal{(ols)}}}
\newcommand\therm{\hat{\th}^{\textnormal{(erm)}}}
\newcommand\hferm{\hf^{\textnormal{(erm)}}}
\newcommand\zols{\zeta^{\textnormal{(ols)}}}

\newcommand\thrid{\tilde{\th}} 
\newcommand\frid{\tilde{f}} 
\newcommand\freg{f^{\textnormal{(reg)}}}
\newcommand\thrlam{\hth^{\textnormal{(ridge)}}} 
\newcommand\hfrlam{\hf^{\textnormal{(ridge)}}} 
\newcommand\thllam{\hth^{\textnormal{(lasso)}}} 
\newcommand\hfllam{\hf^{\textnormal{(lasso)}}} 

\newcommand\flin{f^*_{\textnormal{lin}}}
\newcommand\thlin{\th^{\textnormal{(lin)}}}
\newcommand\Flin{\mathcal{F}_{\textnormal{lin}}}

\renewcommand\Phi{\XX}
\newcommand\demi{\frac{1}{2}}
\newcommand\demic{\fracc{1}{2}}

\newcommand\substa[2]{\substack{#1\\#2}}
\newcommand\substac[2]{{#1\,;\,#2}}
\newcommand\tpsi{\tilde{\psi}}
\newcommand\tzeta{\tilde{\zeta}}
\newcommand\ta{\tilde{a}}
\newcommand\chis{\chi_\sigma}
\newcommand\tchi{\tilde{\chi}}
\newcommand\tchis{\tchi_\sigma}
\newcommand\psis{\psi_\sigma}

\newcommand\tA{\tilde{A}}
\newcommand\tL{\tilde{L}}

\newcommand\hL{\hat{L}}
\newcommand\hcE{\hat{\cE}}
\newcommand\hDe{\hat{\cE}}
\newcommand\cEb{\cE^{\sharp}}
\newcommand\La{L^{\flat}}
\newcommand\cEa{\cE^{\flat}}
\newcommand\Lb{L^{\sharp}}

\newcommand\Pa{P^{\flat}}
\newcommand\Pb{P^{\sharp}}

\newcommand{\V}[1]{\overline{#1}}
\newcommand\bL{\V{L}}

\newcommand\ela{\tilde{\ell}}
\newcommand\hpig{\hpi^{\textnormal{(Gibbs)}}}

\newcommand\sigmb{\phi}
\newcommand{\tR}{\tilde{\cR}}
\newcommand{\logdeps}{\log(4d\eps^{-1})}
\newcommand{\logddeps}{\log(2d^2\eps^{-1})}

\newcommand{\piobs}{\bar{\pi}}

\newcommand\cdd{\mathcal{D}'}

\section*{Introduction}
\addcontentsline{toc}{section}{Introduction}
\subsection*{Our statistical task}
\addcontentsline{toc}{subsection}{Our statistical task}
Let $Z_1=(X_1,Y_1),\dots,Z_n=(X_n,Y_n)$ be $n\ge 2$ pairs of input-output 
and assume that each pair has been independently drawn from the same unknown distribution $P$. Let 
$\X$ denote the input space and let the output space be the set of real numbers $\R$, so that $P$ is a probability distribution on the product space 
$\Z \eqdef \X\times\R$. 
The target of learning algorithms is to predict the output $Y$ associated with an input $X$
for pairs $Z=(X,Y)$ drawn from the distribution $P$. 
The quality of a (prediction) function 
$f:\X\rightarrow\R$ is measured by the least squares {\jyem risk}: 
        $$R(f) \eqdef \undc{\E}{Z\sim P} \bigl\{ [Y-f(X)]^2 
\bigr\} .$$
Through the paper, we assume that the output and all the prediction functions we consider are square integrable.
Let $\cC$ be a closed convex set of $\R^d$, and $\vp_1,\dots,\vp_d$ be $d$ prediction functions. 
Consider the regression model
	\begin{align*}
	\cF= \bigg\{ f_\th=\sum_{j=1}^d \th_j \vp_j ; (\th_1,\dots,\th_d) \in \cC \bigg\}.
	\end{align*}
The best function $f^*$ in $\cF$ is defined by
	\begin{align} \label{eq:fstar}
	f^* =\sum_{j=1}^d \th^*_j \vp_j\in\und{\argmin}{f\in\cF} \,R(f).
	\end{align}
Such a function always exists but is not necessarily unique. Besides it is unknown since the probability generating the data is unknown.

We will study the problem of predicting (at least) as well as function $f^*$. In other words, we want to deduce from the observations
$Z_1,\dots,Z_n$ a function $\hf$ having with high probability a risk bounded by the minimal risk $R(f^*)$ on $\cF$ plus a small remainder term, which is typically of order $d/n$. 
Except in particular settings (e.g., when $\cC$ is a probability simplex\footnote{This corresponds to the convex aggregation problem, which has been widely studied by several authors since the work of Nemirovski and Judisky \cite{Nem98,Jud00}. This particular setting is not the topic of this paper, but our results apply to it, and correspond to the minimax optimal rate for $d\le \sqrt{n}$. For $d>\sqrt{n}$, the minimax optimal rate of convex aggregation is $\sqrt{\fracc{\log(1+ d/\sqrt{n})}n}$, which is not achieved by our procedure.}
and $d\ge \sqrt{n}$), 
it is known that the convergence rate $d/n$ cannot be improved in a minimax sense (see \cite{Tsy03}, and \cite{Yan01b} for related results).

More formally, the target of the paper is to develop estimators $\hf$ for which the excess risk is controlled {\jyem in deviations}, i.e., such that
for an appropriate constant $\kap>0$, for any $\eps>0$, with probability at least $1-\eps$,
	\beglab{eq:devtarget}
	R(\hf) - R(f^*) \le \kap \frac{d+\logeps}{n}.
	\endlab
Note that by integrating the deviations (using the identity 
$\E W = \int_0^{+\infty} \P(W > t) dt$ 
which holds true for any nonnegative random variable $W$), 
Inequality \eqref{eq:devtarget} implies 
	\beglab{eq:exptarget}
	\E R(\hf) - R(f^*) \le \kap \frac{d+1}{n}.
	\endlab
In this work, we do not assume that the function 
	\[
	\freg:x \mapsto \E[ Y | X=x],
	\]
which minimizes the risk $R$ among all possible measurable functions,
belongs to the model $\cF$. So we might
have $f^* \neq \freg$ and in this case, bounds of the form
	\beglab{eq:nottarget}
	\E R(\hf) - R(\freg) \le C [R(f^*)-R(\freg)] + \kap \frac{d}{n},
	\endlab
with a constant $C$ larger than $1$ do not even ensure that $\E R(\hf)$ tends to $R(f^*)$ when $n$ goes to infinity.
This kind of bounds with $C>1$ have been developed to analyze nonparametric estimators using linear approximation spaces,
in which case the dimension $d$ is a function of $n$ chosen so that the bias term $R(f^*)-R(\freg)$ has the order $d/n$ of the estimation term
(see \cite{Gyo04} and references within). Here we intend to assess the
generalization ability of the estimator even when the model is misspecified 
(namely when $R(f^*) > R(\freg)$).
Moreover we do not assume either that $Y-\freg(X)$ and $X$ are independent.

\textbf{Notation.} \hspace*{2mm}
When $\cC=\R^d$, 
the function $f^*$ and the space $\cF$ will be written $\flin$ and $\Flin$ to emphasize that 
$\cF$ is the whole linear space spanned by $\vp_1,\dots,\vp_d$:
\[  \Flin = \Span \{\vp_1,\dots,\vp_d\} \text{\qquad and\qquad} \flin\in\und{\argmin}{f\in\Flin} R(f).\]
The Euclidean norm will simply be written as $\|\cdot\|$, and $\langle \cdot,\cdot\rangle$
will be its associated inner product.
We will consider the vector valued function
$\vp : \C{X} \rightarrow \B{R}^d$ defined
by $\vp(X) = \bigl[ \vp_k(X) \bigr]_{k=1}^d$, so that for any $\th\in\Theta$, we have
  \[
  f_\th(X) = \langle \th , \vp(X) \rangle.
  \]
The Gram matrix is the $d\times d$-matrix $Q=\B{E} \bigl[ \vp(X) \vp(X)^T\bigr]$, and its smallest and largest eigenvalues will 
respectively be written as $q_{\min}$ and $q_{\max}$. 
The empirical risk of a function $f$ is
$$
r(f) = \frac{1}{n} \sum_{i=1}^n \bigl[ f(X_i) - 
Y_i \bigr]^2
$$
and for $\lam\ge 0$, the ridge regression estimator on $\cF$ is defined by
$\hfrlam=f_{\thrlam}$ with
$$
\thrlam \in \arg \min_{\th \in \Theta} 
r(f_{\theta}) + \lambda \lVert \theta \rVert^2,
$$
where $\lam$ is some nonnegative real 
parameter. In the case when $\lambda = 0$, 
the ridge regression $\hfrlam$ is nothing
but the empirical risk minimizer $\hferm$.
In the same way, we introduce the optimal ridge function  
optimizing the expected ridge risk: $\frid=f_{\thrid}$ with
\beglab{eq:frid}
\thrid \in \arg \min_{\theta \in \Theta} \big\{ R(f_{\theta}) 
+ \lambda \lVert \theta \rVert^2 \big\}.
\endlab
Finally, let $Q_{\lambda} = Q + \lambda I$ be the ridge regularization 
of $Q$, where $I$ is the identity matrix.

\subsection*{Outline and contributions}
\addcontentsline{toc}{subsection}{Outline and contributions}

The paper is organized as follows.
Section~\ref{sec:lit} is a survey on risk bounds in linear least squares regression.
Theorems \ref{th:alqa} and \ref{th:bmnew} are the results which come closer to our target.
Section~\ref{sec:main} presents our main result on linear least squares regression. 
Section~\ref{sec:gen} gives risk bounds for general 
loss functions from which the results 
of Section~\ref{sec:main} are derived. 
Appendix \ref{sec:lb} shows that \eqref{eq:devtarget}
cannot hold under the only
assumption that the variance of $Y$ is finite, even in the favorable
situation where $f^{(reg)}$ belongs to $\cF$.

The main contribution of this paper is to show that an appropriate shrinkage estimator involving truncated differences of losses has an excess risk of order $d/n$ (without a logarithmic factor as it appears in numerous works), concentrating exponentially, which does not 
degrade when the matrix $Q$ is ill-conditioned or when some ratio of $L^2$ and $L^\infty$ norms behaves badly or when the output distribution is heavy-tailed.
Our results tend to say that shrinkage and truncation lead to more robust algorithms
when we consider robustness with respect to the distribution of the noise, and not
to a potential contamination of the training data by input-output pairs not generated by $P$.

\section{Variants of known results} \label{sec:lit}

\subsection{Ordinary least squares and empirical risk minimiza\-tion} \label{sec:ols}

The ordinary least squares estimator is the most standard method in linear least squares regression. 
It minimizes the empirical risk 
$$
	r(f) = \frac{1}{n} \sum_{i=1}^n [Y_i-f(X_i)]^2,
$$
among functions in $\Flin$ and produces
$$
	\hfols= \sum_{j=1}^d \thols_j \vp_j,
$$
with $\thols=[\thols_j]_{j=1}^d$ a column vector satisfying
	\begarlab{eq:thols}
	\Phi^T\Phi \, \thols = \Phi^T \YY,
	\endarlab
where $\YY=[Y_j]_{j=1}^n$ and $\Phi=(\vp_j(X_i))_{ 1\le i \le n, 1\le j \le d}$. 
It is well-known that 
\bi
\item the linear system \eqref{eq:thols} has at least one solution, and in fact, the set of solutions is exactly $\{\Phi^+ \YY + u ; u \in \text{ker}\, \Phi\}$;
where $\Phi^+$ is the Moore-Penrose pseudoinverse of $\Phi$ and $\text{ker} \,\Phi$ is the kernel of the linear operator $\Phi$.
\item $\Phi \, \thols$ is the (unique) orthogonal projection of the vector $\YY\in\R^n$
on the image of the linear map $\Phi$;
\item if $\sup_{x\in\X} \Var(Y |X=x) = \sigma^2 < +\infty$, we have
(see \cite[Theorem 11.1]{Gyo04}) for any $X_1,\dots,X_n$ in $\X$,
	\begin{multline} \label{eq:fixeddesign}
	\E \bigg\{ \frac{1}{n} \sum_{i=1}^n \big[\hfols(X_i) - \freg(X_i)\big]^2 \bigg| X_1,\dots,X_n\bigg\}\\
		\qquad - \und{\min}{f\in\Flin} \frac{1}{n} \sum_{i=1}^n \big[f(X_i) - \freg(X_i)\big]^2 
		\le \sigma^2 \frac{\text{rank}(\Phi)}{n} \le \sigma^2 \frac{d}{n},
	\end{multline}
where we recall that $\freg: x\mapsto \E[ Y | X=x]$ is the optimal regression function, and 
that when this function belongs to $\Flin$ (i.e., $\freg=\flin$), the minimum term in \eqref{eq:fixeddesign} vanishes; 
\item from Pythagoras' theorem for the (semi)norm $W \mapsto \sqrt{\E W^2}$ on the space of the square integrable random variables, 
\begin{multline} \label{eq:pyt}
	R(\hfols) - R(\flin) \\ = \E \big[ \hfols(X) - \freg(X) \big| Z_1,\dots,Z_n \big]^2 - \E \big[ \flin(X) - \freg(X) \big]^2.
\end{multline}
\ei
The analysis of the ordinary least squares often stops at this point in classical statistical textbooks.
(Besides, to simplify, the strong assumption $\freg=\flin$ is often made.)
This can be misleading since Inequality \eqref{eq:fixeddesign} does not imply a $d/n$ upper bound on the risk
of $\hfols$. Nevertheless the following result holds \cite[Theorem 11.3]{Gyo04}.
\begin{thm}\label{th:weakols}
If $\,\sup_{x\in\X} \Var(Y |X=x) = \sigma^2 < +\infty$ 
and 
$$
\|\freg\|_\infty = \sup_{x\in\X} |\freg(x)| \le H
$$ for some $H>0$, then
the truncated estimator $\hfols_H=(\hfols\wedge H)\vee -H$ satisfies 
	\begin{multline} \label{eq:weakols}
	\E R(\hfols_H) - R(\freg) \le 8 [R(\flin)-R(\freg)] + \kap \frac{(\sigma^2\vee H^2)d\log n }{n} 
	\end{multline}
for some numerical constant $\kap$.
\end{thm}

Using PAC-Bayesian inequalities, 
Catoni \cite[Proposition 5.9.1]{Cat01} has proved a different type of results on the generalization ability of $\hfols$.

\begin{thm} \label{th:oc}
Let $\cF'\subset\Flin$ be such that for some positive constants $a,M,M'$:
\bi
\item $\text{there exists } f_0 \in \cF' \text{ s.t. for any } x\in\X,$ 
  \[
  \E \Bigl\{ 
  \exp \Bigl[ a \bigl\lvert Y-f_0(X) \bigr\rvert \Bigr]\, \Big| \,X=x \Bigr\} \le M;
  \]
\item $\text{for any } f_1,f_2 \in \cF', \sup_{x\in\X} |f_1(x)-f_2(x)|\le M'$.
\ei
Let $Q=\B{E} \bigl[ \vp(X) \vp(X)^T\bigr]$
and $\hat{Q} = \bigl[ \frac{1}{n} \sum_{i=1}^n \vp(X_i) \vp(X_i)^T \bigr]$
be respectively the expected and empirical Gram matrices. If $\det Q\neq 0$, then
there exist positive constants $C_1$ and $C_2$ (depending only on $a$, $M$ and $M'$) such that
with probability at least $1-\eps$, 
as soon as
 	\beglab{eq:cond}
	\bigg\{
	f\in\Flin:r(f) \le r(\hfols) + C_1 \frac{d}{n}
	\bigg\} \subset \cF',
	\endlab
we have
	\[
	R(\hfols) - R(\flin) \le C_2 \frac{d+\logeps+\log(\frac{\det \hat{Q}}{\det Q})}{n}.
	\]
\end{thm}

This result can be understood as follows. 
Let us assume we have some prior knowledge suggesting 
that $\flin$ belongs to the interior of a set 
$\cF'\subset\Flin$ (e.g., a bound
on the coefficients of the expansion of $\flin$ as a linear
combination of $\vp_1, \dots, \vp_d$). 
It is likely that \eqref{eq:cond} holds, and it is indeed proved in 
Catoni \cite[section 5.11]{Cat01} that the probability that it does not
hold goes to zero exponentially fast with $n$ in the case when $\cF'$
is a Euclidean ball. 
If it is the case, then we know that the excess risk is of order $d/n$
up to the unpleasant ratio of determinants, which, 
fortunately, almost surely tends to $1$ as $n$ goes to infinity.

%
%

%
By using \emph{localized} PAC-Bayes inequalities introduced in Catoni 
\cite{Cat03b, Cat05},
one can derive from Inequality (6.9) and Lemma 4.1 of Alquier \cite{Alq08} the following result.

\begin{thm} \label{th:alqa}
Let $q_{\min}$ be the smallest eigenvalue of the Gram matrix $Q=\B{E} \bigl[ \vp(X) \vp(X)^T\bigr]$.
Assume that there exist a function $f_0 \in \Flin$ and positive constants $H$ and $C$ such that 
	\[
	\| \flin - f_0 \|_{\infty}  \le H.
	\]
and $|Y| \le C$ almost surely.

Then for an appropriate randomized estimator requiring the knowledge of $f_0$, $H$ and $C$, for any $\eps>0$
with probability at least $1-\eps$ $\wrt$ the distribution generating
the observations $Z_1,\dots,Z_n$ and the randomized prediction function $\hf$, we have
    \beglab{eq:alqa}
    R(\hat{f}) - R(\flin) \le \kap (H^2 +C^2) \frac{d \log( 3q_{\min}^{-1} ) + 
        \log \bigl[ (\log n) \eps^{-1} \bigr] }{n},
    \endlab
for some $\kap$ not depending on $d$ and $n$.
\end{thm}

Using the result of \cite[Section 5.11]{Cat01},
one can prove that Alquier's result still holds for $\hf= \hfols$,
but with $\kap$ also depending on the determinant of the product matrix $Q$.
The $\log[ \log(n) ]$ factor is unimportant and could be removed in
the special case quoted here (it comes from a union bound on a grid
of possible temperature parameters, whereas the temperature could
be set here to a fixed value). The result
differs from Theorem \ref{th:oc} essentially by the fact that
the ratio of the determinants of the empirical and expected product
matrices has been replaced by the inverse of the smallest eigenvalue
of the quadratic form $\theta \mapsto R(\sum_{j=1}^d \theta_j\vp_j) - R(\flin)$.  
In the case when the expected Gram matrix is known,
(e.g., in the case of a fixed design, and also in the slightly
different context of transductive inference), this
smallest eigenvalue can be set to one by choosing the 
quadratic form $\theta \mapsto R(f_{\theta}) - R(\flin)$
to define the Euclidean metric on the parameter space.

Localized Rademacher complexities \cite{Kol06,BarBouMen05} allow
to prove the following property of the empirical risk minimizer.

\begin{thm} \label{th:bbm}
Assume that the input representation $\vp(X)$, the set of parameters and the output $Y$ are
almost surely bounded, i.e., for some positive constants $H$ and $C$,
  \[
  \sup_{\th\in \Theta} \|\th\| \le 1
  \]
  \[
  \ess\sup \|\vp(X)\| \le H,
  \]
and 
  \[
  |Y| \le C \quad \textnormal{a.s.}.
  \]
Let $\nu_1\ge\dots\ge\nu_d$ be the eigenvalues of the  Gram matrix $Q=\B{E} \bigl[ \vp(X) \vp(X)^T\bigr]$.
The empirical risk minimizer satisfies for any $\eps>0$,
with probability at least $1-\eps$:
	\begin{align*}
	R( \hferm ) - R(f^*) & \le \kap (H+C)^2 
	  \frac{\und{\min}{0\le h\le d} \Big( h+\sqrt{ \frac{n}{(H+C)^2} 
	  \sum_{i>h}\nu_i}\Big)+\logeps}{n}\\
  & \le \kap (H+C)^2 \frac{\textnormal{rank}(Q) +\logeps}{n},
	\end{align*}
where $\kap$ is a numerical constant.
\end{thm}
\begin{proof}
The result is a modified version of Theorem 6.7 in \cite{BarBouMen05} 
applied to the linear kernel $k(u,v)=\langle u , v \rangle/(H+C)^2$.
Its proof follows the same lines as in Theorem 6.7 {\em mutatis mutandi}:
Corollary 5.3 and Lemma 6.5 should be used as intermediate steps 
instead of Theorem 5.4 and Lemma 6.6, 
the nonzero eigenvalues of the integral operator induced 
by the kernel being the nonzero eigenvalues of $Q$.
\end{proof}

When we know that the target function $\flin$ is inside some $L^\infty$ ball, it is natural to consider the empirical risk minimizer on this ball. This allows
to compare Theorem \ref{th:bbm} to excess risk bounds with respect to
$\flin$.

Finally, from the work of Birg\'e and Massart \cite{BirMas98}, we may derive the following risk bound for the empirical risk minimizer on a $L^\infty$ ball (see Appendix \ref{sec:bm}).

\begin{thm} \label{th:bmnew}
Assume that $\cF$ has a diameter upper bounded by $H$ for the $L^\infty$-norm, i.e., for any 
$f_1,f_2$ in $\cF$, $\sup_{x\in\X} |f_1(x)-f_2(x)| \le H$ and there exists a function
$f_0 \in\cF$ satisfying the exponential moment condition:
	\beglab{eq:expmom}
	\text{for any } x\in\X, \quad 
  \E \Bigl\{ \exp \Bigl[ A^{-1} \bigl\lvert Y-f_0(X) \bigr\rvert \Bigr]\, \Big| \,X=x \Bigr\} \le M,
	\endlab
for some positive constants $A$ and $M$. Let
	\[
	\cR = \und{\inf}{\phi_1,\dots,\phi_d} \, \und{\sup}{\th\in \R^d - \{0\}} 
		\frac{\|\sum_{j=1}^d \th_j \phi_j\|_\infty^2}{\|\th\|_\infty^2}
	\]
where the infimum is taken with respect to all possible orthonormal basis of $\cF$ for 
the dot product $\lan f_1,f_2 \ran = \E f_1(X)f_2(X)$ (when the set $\cF$ admits no basis with exactly $d$ functions, we set $\cR=+\infty$).
Then the empirical risk minimizer satisfies for any $\eps>0$,
with probability at least $1-\eps$:
	\[
	R( \hferm ) - R(f^*) \le \kap (A^2 + H^2) \frac{d \log[2+(\cR /n)\wedge (n/d)] +\logeps}{n},
	\]
where $\kap$ is a positive constant depending only on $M$.
\end{thm}

This result comes closer to what we are looking for: it gives exponential deviation inequalities of order at worse 
$\fracc{d \log(n/d)}{n}$.
It shows that, even if the Gram matrix $Q$
has a very small eigenvalue, there is an algorithm satisfying a convergence rate of order 
$\fracc{d \log(n/d)}{n}$. With this respect, this result is stronger than Theorem \ref{th:alqa}. However there
are cases in which the smallest eigenvalue of $Q$ is of order $1$, while $\cR$ is large (i.e., $\cR \gg n$). In these cases, 
Theorem \ref{th:alqa} does not contain the logarithmic factor which appears in Theorem \ref{th:bmnew}.

\subsection{Projection estimator} \label{sec:proj}

When the input distribution is known, an alternative to the ordinary least squares estimator is 
the following projection estimator. One first finds an orthonormal 
basis of $\Flin$ for the dot product $\lan f_1,f_2 \ran = \E f_1(X)f_2(X)$, 
and then uses the projection estimator on this basis. Specifically, if $\phi_1,\dots,\phi_d$ form an orthonormal basis of $\Flin$, then the projection estimator on this basis is:
	$$
	\hfproj= \sum_{j=1}^d \thproj_j \phi_j,
	$$
with 
	$$
	\thproj= \frac{1}{n} \sum_{i=1}^n Y_i \phi_j(X_i).
	$$
The following excess risk bound of order $d/n$ for this estimator is Theorem 4 in \cite{Tsy03} up to minor changes in the assumptions.

\begin{thm} \label{th:proj}
If $\sup_{x\in\X} \Var(Y |X=x) = \sigma^2 < +\infty$
and \[\|\freg\|_\infty = \sup_{x\in\X} |\freg(x)| \le H < +\infty,\] then we have
	\beglab{eq:tsy}
	\E R(\hfproj) - R(\flin) \le (\sigma^2+H^2) \frac{d}{n}.
	\endlab
\end{thm}

\subsection{Penalized least squares estimator}

It is well established that parameters of the ordinary least squares estimator are numerically unstable, and that the phenomenon can be corrected
by adding an $L^2$ penalty (\cite{Lev44,Ril55}). This solution has been labeled ridge regression in statistics (\cite{Hoe62}), and consists 
in replacing $\hfols$ by $\hfrlam=f_{\thrlam}$ with
	\[
	\thrlam \in \und{\argmin}{\th\in\R^d} 
	  \bigg\{ r(f_\th) + \lam \sum_{j=1}^d \theta_j^2 \bigg\},
	\]
where $\lam$ is a positive parameter. 
The typical value of $\lam$ should be small to avoid excessive shrinkage of the coefficients, but not too small in
order to make the optimization task numerically more stable.

Risk bounds for this estimator can be derived from general results concerning penalized least squares on reproducing kernel Hilbert spaces
(\cite{CapVit07}), but as it is shown in Appendix \ref{sec:capvit}, this ends up with complicated results having
the desired $d/n$ rate only under strong assumptions. 

Another popular regularizer is the $L^1$ norm. This procedure is known 
as Lasso \cite{Tib94} and is defined by
	\[
	\thllam \in \und{\argmin}{\th\in\R^d} 
		\bigg\{ r(f_\th) + \lam \sum_{j=1}^d |\theta_j| \bigg\}.
	\]
As the $L^2$ penalty, the $L^1$ penalty shrinks the coefficients. The difference is that for coefficients 
which tend to be close to zero, the shrinkage makes them equal to zero. This allows to select relevant 
variables (i.e., find the $j$'s such that $\th^*_j\neq 0$). 
If we assume that the regression function $\freg$ is a linear combination of only $d^* \ll d$ variables/functions $\vp_j$'s,
the typical result is to prove that the risk of the Lasso estimator for $\lam$ of order $\sqrt{\fracl{\log d}{n}}$
is of order $\fracl{d^* \log d}{n}$. Since this quantity is much smaller than $d/n$, this makes a huge improvement (provided that
the sparsity assumption is true). This kind of results usually requires strong conditions on the
eigenvalues of submatrices of $Q$, essentially assuming that the functions 
$\vp_j$ are near orthogonal. We do not know to which extent these 
conditions are required. However, if we do not consider the specific algorithm of Lasso, but the model selection approach developed in \cite{Alq08}, one can change these conditions into a single condition concerning only the minimal eigenvalue of the submatrix of $Q$ corresponding to relevant variables. In fact, we will see that even this condition can be removed.

\subsection{Conclusion of the survey} 

Previous results clearly leave room to improvements. 
The projection estimator requires the unrealistic assumption that the input distribution is known,
and the result holds only in expectation. Results using $L^1$ or $L^2$ regularizations require strong assumptions, in particular on the eigenvalues of (submatrices of) $Q$.
Theorem \ref{th:weakols} provides a $(d \log n)/n$ convergence rate only when the 
$R(\flin)-R(\freg)$ is at most of order $(d \log n)/n$. 
Theorem \ref{th:oc} gives a different type of guarantee: the $d/n$ is 
indeed achieved, but the random ratio of determinants appearing 
in the bound may raise some eyebrows and forbid an explicit
computation of the bound and comparison with other bounds.
Theorem \ref{th:alqa}
seems to indicate that the rate of convergence will
be degraded when the Gram matrix $Q$ is unknown and ill-conditioned. 
Theorem \ref{th:bbm} does not put any assumption on $Q$ to reach the $d/n$ rate, 
but requires particular boundedness constraints on the output.
Finally, Theorem \ref{th:bmnew} comes closer to what we are looking for. Yet there is still an unwanted logarithmic factor, and 
the result holds only when the output has uniformly bounded 
conditional exponential moments, which as we will show is not necessary.

Our recent work \cite{AudCat110}
provides a risk bound for ridge regression showing the benefit on the
effective dimension of the shrinkage parameter $\lam$ and being of order $d/n$ (without logarithmic factor).
The work \cite{AudCat110} also proposes a robust estimator for linear least squares, which satisfies a $d/n$ excess risk bound without logarithmic factor, but with constants involving several kurtosis coefficients. As discussed in Section 3.2 of \cite{AudCat110}, depending on the basis functions and the distribution $P$, these kurtosis coefficients typically behave either as numerical constants or $\sqrt{d}$ (but worse non-asymptotic behaviors of these constants can also occur).

Finally, several works, and in particular those cited in Section \ref{sec:ols}, have considered the 
problem of model selection where several linear spaces are simultaneously considered, and the goal is to predict as well as the
best function in the union of the linear spaces.
Only a few of them considered the case of outputs having only finite conditional moments (and not finite conditional exponential moments).
This is the case of \cite{Bar00} in the fixed design setting and \cite{Weg03} in the random design setting. The excess risk bounds 
there are typically of order $d/n$ with $d$ the dimension of the ``best'' linear space, but holds in expectation and essentially when
the optimal regression function $\freg$ belongs to the union of linear spaces. 

\section{A simple tight risk bound for a sophisticated PAC-Bayes algorithm} \label{sec:main}

In this section, we provide a sophisticated estimator, 
having a simple theoretical excess risk bound, 
with neither a logarithmic factor, nor complex constants involving the conditioning of $Q$, kurtosis coefficients or some
geometric quantity characterizing the relation between $L^2$ and $L^\infty$-balls.

We consider that the set $\cC$ is bounded so that
we can define the ``prior'' distribution $\pi$ as the uniform distribution on $\cF$ 
(i.e., the one induced by the Lebesgue distribution on $\cC\subset \R^d$ 
renormalized to get $\pi(\cF)=1$).  
Let $\lam>0$ and 
	\[W_i(f,f') = \lam \bigl\{ \bigl[ Y_i-f(X_i)\bigr]^2 - \bigl[
Y_i-f'(X_i) \bigr]^2 \bigr\}. \]
Introduce 
	\beglab{eq:hcEls}
	\hcE(f) = \log \int \frac{\pi(df')}{\prod_{i=1}^n [ 1 - W_i(f,f') + \frac{1}{2} W_i(f,f')^2]}.
	\endlab
We consider the ``posterior'' distribution $\hpi$ on the set $\cF$ with density:
	\beglab{eq:hpils}
	\frac{d \hpi}{d \pi} (f) = \frac{\ds \exp [-\hcE(f)]}{\ds \tint
\exp [-\hcE(f') ] \pi(df')}.
	\endlab
To understand intuitively why this distribution concentrates on functions with low risk,
one should think that when 
$\lam$ is small enough, $1-W_i(f,f') + \frac{1}{2} W_i(f,f')^2$
is close to $e^{-W_i(f,f')}$, and consequently
  \[
  \hcE(f) \approx \lam \sum_{i=1}^n [Y_i-f(X_i)]^2 + \log \int \pi(df') \exp 
\Bigl\{-\lam \sum_{i=1}^n \bigl[Y_i-f'(X_i)\bigr]^2 \Bigr\},
  \]
and
  \[
  \frac{ d \hpi}{d \pi} (f) \approx 
\frac{\exp \{ -\lam \sum_{i=1}^n [Y_i-f(X_i)]^2 \} }{\int \exp 
\{ -\lam \sum_{i=1}^n [Y_i-f'(X_i)]^2 \} \pi(df')}\,.
  \]
The following theorem gives a $d/n$ convergence rate for the randomized algorithm which draws the prediction function from $\cF$
according to the distribution $\hpi$.

\begin{thm} \label{th:main}
Assume that $\cF$ has a diameter upper bounded by $H$ for the $L^\infty$-norm:
	\beglab{eq:hh}
	\sup_{f_1,f_2\in \cF,x\in\X} |f_1(x)-f_2(x)| \le H
	\endlab
and that, for some $\sigma>0$, 
	\beglab{eq:sigma}
	\sup_{x\in\X} \E\big\{[Y-f^*(X)]^2 \big| X=x\big\} \le \sigma^2 < +\infty.
	\endlab
Let $\hf$ be a prediction function drawn from the distribution $\hpi$
defined in \eqref{eq:hpils} and depending on the parameter $\lam>0$.
Then for any $0<\eta'<1-\lam (2\sigma+H)^2$ and $\eps>0$, with probability (with respect to the distribution $P^{\otimes n} \hpi$ generating
the observations $Z_1,\dots,Z_n$ and the randomized prediction function $\hf$) at least $1-\eps$, we have
	\[
	R( \hf ) - R(f^*) \le (2\sigma+H)^2 \, \frac{C_1 d + C_2 \log(2\eps^{-1})}{n} 
	\]
with 
	\[
	C_1= \frac{\log(\frac{(1+\eta)^2}{\eta'(1-\eta)}) }{\eta (1-\eta-\eta')} \quad \text{and} \quad
	C_2=\frac{2}{\eta(1-\eta- \eta')} \quad \text{and} \quad \eta= \lam (2\sigma+H)^2.
	\]
In particular for $\lam=0.32 (2\sigma+H)^{-2}$ and $\eta'=0.18$, we get
	\[
	R( \hf ) - R(f^*) \le (2\sigma+H)^2 \, \frac{16.6 \, d 
    		+ 12.5 \log(2\eps^{-1})}{n}.	
	\]
Besides if $f^*\in\argmin_{f\in\Flin} R(f)$, then with
probability at least $1-\eps$, we have
	\[
	R( \hf ) - R(f^*) \le (2\sigma+H)^2 \, \frac{8.3 \, d 
    		+ 12.5 \log(2\eps^{-1})}{n}.	
	\]
\end{thm}

\begin{proof}
This is a direct consequence of Theorem \thmref{th:v2c},
Lemma \thmref{le:complexity} and Lemma \thmref{le:v2}.
\end{proof}

If we know that $\flin$ belongs to some bounded ball in 
$\Flin$, then one can define a bounded $\cF$ as this ball, use the previous 
theorem and obtain an excess risk bound with respect to
$\flin$.


\begin{rmk} \label{rmk:main}
Let us discuss this result. On the positive side, we have a $d/n$ convergence rate in expectation and in deviations.
It has no extra logarithmic factor. It does not require any particular assumption on the smallest eigenvalue of the covariance matrix.
To achieve exponential deviations, a uniformly bounded second moment of the output knowing the input is surprisingly sufficient: we do not
require the traditional exponential moment condition on the output.
Appendix \thmref{sec:lb} argues that the uniformly bounded conditional second moment assumption 
cannot be replaced with just a bounded second moment condition.

On the negative side, the estimator is rather complicated. 
With nowadays computers and numerical methods, it seems impossible to get a good approximation of it even when the dimension $d$ is small. Nevertheless, in presence of a heavy-tailed noise distribution, it can be a way to move from the empirical risk minimizer (which is the baseline estimator for linear regression) in the right direction (that is in a direction in which one can find an estimator having a smaller risk than the one of the empirical risk minimizer).
When the target is to predict as well as the best linear combination $\flin$ up to a small additive term, the estimator requires the knowledge 
of a $L^\infty$-bounded ball in which $\flin$ lies and an upper bound
on $\sup_{x\in\X} \E\big\{[Y-\flin(X)]^2 \big| X=x\big\}$.
The looser this knowledge is, the bigger the constant in front of $d/n$ is.
Note that the possible lack of knowledge of $H$ and $\sigma$ call for 
a model selection algorithm, which goes beyond the scope of this work.
In practice, a careful application of (cross-)validation ideas would probably be sufficient to select these parameters.
\end{rmk}

\begin{rmk}
The proposed randomized estimator is more complex than the classical Gibbs estimator (that is the one with exponential weights involving the empirical risk). Even if the paper does not prove it, (we believe that) the classical Gibbs estimator cannot be robust to heavy-tailed noise. This belief is motivated by the same arguments as the ones used in \cite{Cat09} to show the absence of robustness of the empirical mean estimator. In absence of heavy-tailed noise, the classical Gibbs estimator satisfies a similar result to Theorem \ref{th:main}, given in Theorem \ref{th:v1c}.

Our randomized algorithm consists in drawing the prediction function according to $\hpi$.
As usual, by convexity of the loss function, the risk of the deterministic estimator $\hf_{\text{determ}} = \int f \hpi(df)$
satisfies $R(\hf_{\text{determ}}) \le \int R(f) \hpi(df)$, so that, after some 
computations, one can prove that 
for any $\eps>0$, with probability at least $1-\eps$:
	\[
	R( \hf_{\text{determ}} ) - R(\flin) \le \kap (2\sigma+H)^2\frac{d + \log(\eps^{-1})}{n},
	\]
for some appropriate numerical constant $\kap>0$. 
\end{rmk}

\begin{rmk}
We consider a ``prior'' distribution $\pi$, which is a uniform distribution on $\cF$.  
In presence of sparsity (when only a small number of the coefficients $\theta^*_j$ in \eqref{eq:fstar} are nonzero), alternative prior distributions (of Laplace form) are useful in fixed design regression \cite{DalTsy08,DalTsy11b,AlqLou11} and in the random design scenario \cite{DalTsy11,AlqLou11}. When the coefficient vector $\th^*$ is non-sparse (which is not the focus of these works), the latter papers prove a $\frac{d \log n}n$ risk bound when the noise distribution admits at least sub-exponential tails.
\end{rmk}

\begin{rmk}
Theorem \ref{th:main} expresses boundedness in terms of the $L^\infty$ diameter of the 
set of functions $\cF$. Besides, \eqref{eq:sigma} implies that the function
$\freg:x \mapsto \E[ Y | X=x]$ satisfies $\freg(X)-f^*(X)\le \sigma$ almost surely.
By using Lemma \thmref{le:v3} instead of Lemma \thmref{le:v2},
Theorem \ref{th:main} still holds without assuming \eqref{eq:hh} and \eqref{eq:sigma},
when replacing $(2\sigma+H)^2$ with 
  \begin{align*}
  V= \bigg[ 2& \sqrt{\sup_{f\in\Flin: \E[f(X)^2]=1}\E\big(f(X)^2[Y-f^*(X)]^2\big)}\\
    & \qquad + \sqrt{\sup_{f',f''\in\cF} \E\big([f'(X)-f''(X)]^2\big)} \sqrt{\sup_{f\in\Flin: \E[f(X)^2]=1}\E\big[f(X)^4\big]} \bigg]^2.
  \end{align*}
The quantity $V$ is finite when simultaneously, $\Theta$ is bounded, and for any $j$ in $\{1,\dots,d\}$, the
quantities $\E\big[\vp_j^4(X)\big]$ and $\E\big\{\vp_j(X)^2[Y-f^*(X)]^2\big\}$ 
are finite.    
\end{rmk}

\section{A generic localized PAC-Bayes approach} \label{sec:gen}

\subsection{Notation and setting}

In this section, we drop the restrictions of the linear least squares setting considered so far
in order to focus on the ideas underlying the estimator and the results presented in Section~\ref{sec:main}. 
To do this, we consider that the loss incurred by predicting $y'$ while the correct output is $y$ is $\ela(y,y')$
(and is not necessarily equal to $(y-y')^2$).
The quality of a (prediction) function $f:\X\rightarrow\R$ is measured by its risk 
	\[R(f) = \E \bigl\{ \ela\bigl[Y,f(X) \bigr] \bigr\}.\]
We still consider the problem of predicting (at least) as well as the best function 
in a given set of functions $\cF$ (but $\cF$ is not necessarily a subset of a finite dimensional linear space).
Let $f^*$ still denote a function minimizing the risk among functions in $\cF$: $f^* \in\undc{\argmin}{f\in\cF} R(f)$.
For simplicity, we assume that it exists. 
The excess risk is defined as 
	\[\bR(f) = R(f) - R(f^*).\]

Let $\ell: \Z\times\cF\times\cF \ra \R$ be a function 
such that $\ell(Z,f,f')$ represents\footnote{While the natural choice in the least squares setting is
$\ell((X,Y),f,f')=[Y-f(X)]^2 - [Y-f'(X)]^2$, we will see that
for heavy-tailed outputs, it is preferable to consider the following soft-truncated
version of it, up to a scaling factor $\lam>0$: $\ell((X,Y),f,f')=T\big(\lam\big[(Y-f(X))^2 - (Y-f'(X))^2\big]\big)$,
with $T(x)=-\log(1-x+x^2/2).$ Equality \myeq{eq:alih} corresponds to \myeq{eq:hcEls} with this choice of function $\ell$
and for the choice $\pis=\pi$.
}
how worse $f$ predicts than 
$f'$ on the data $Z$. 
Let us introduce the real-valued random processes $L: (f,f') \mapsto \ell(Z,f,f')$ and 
$L_i: (f,f') \mapsto \ell(Z_i,f,f')$, where $Z,Z_1,\dots,Z_n$ 
denote i.i.d. random variables with distribution $P$.

Let $\pi$ and $\pis$ be two (prior) probability distributions on $\cF$. 
We assume the following integrability condition.

{\bf Condition I.} \label{cond:i}
For any $f\in\cF$, we have 
\begin{align}
	\label{eq:intega}
	\int \E \bigl\{ \exp [ L(f,f') ]  \bigr\}^n \pis(df') & < +\infty,\\
\text{and } \quad
	\label{eq:integb}
	\int \frac{ \pi(df) }{\int \E \bigl\{ 
\exp [ L(f,f') ] \big\}^n \pis(df')} & < +\infty.
\end{align}
We consider the real-valued processes
	\begin{align}
	\hL(f,f')& = \sum_{i=1}^n L_i(f,f'),\\
	\hcE(f) & =\log \int \exp \bigl[ \hL(f,f') \bigr] \, \pis(df') \label{eq:alih},\\
	\La(f,f')& = - n \log \Bigl\{ \E \Bigl[ \exp\bigl(-L(f,f')\bigr) \Bigr] \Bigr\}, \\
	\Lb(f,f')& = n \log \Bigl\{ \E \Bigl[ \exp \bigl( L(f,f')\bigr) \Bigr] \Bigr\}, \\ 
\text{and } \quad	\cEb(f)& = \log \biggl\{ \int \exp \bigl[ \Lb(f,f') \bigr] 
\, \pis(df') \biggr\}. \label{eq:alib}
	\end{align}	
Essentially, the quantities $\hL(f,f')$, $\La(f,f')$ and $\Lb(f,f')$ represent 
how worse is the prediction from $f$ than from $f'$ with respect to the training data or in expectation. 
By Jensen's inequality, we have 
	\beglab{eq:interv}
	\La \le n \E (L) = \E (\hL) \le \Lb.
	\endlab
The quantities $\hcE(f)$ and $\cEb(f)$ should be understood as some kind of (empirical or expected) 
excess risk of the prediction function~$f$ with respect to an implicit reference induced by the integral over $\cF$.

For a distribution $\rho$ on $\cF$ absolutely continuous $\wrt$  $\pi$, let 
$\ds \frac{d \rho}{d \pi}$ denote the 
density of $\rho$ $\wrt$ $\pi$. For any real-valued (measurable) function $h$ defined on $\cF$ such that
$\int \exp [h(f)] \pi(df)<+\infty$, we define the distribution 
$\pi_{h}$ on $\cF$ by its density:
	\beglab{eq:pih}
	\frac{d \pi_{h}}{d \pi}(f) = 
\frac{\ds \exp [ h(f)]}{\ds \tint \exp [h(f')] \pi(df')}.
	\endlab
We will use the posterior distribution:
	\beglab{eq:hpi}
	\frac{d \hpi}{d \pi} (f) = 
\frac{ d \pi_{-\hcE}}{d \pi}(f) = 
\frac{\exp[-\hcE(f)]}{\int \exp[-\hcE(f')] \pi(df')}.
	\endlab	
Finally, for any $\be \ge 0$, we will use the following measures of the size (or complexity) of $\cF$ around 
the target function:
    \[
    \I^*(\be) = - \log \Bigl\{ \tint \exp \bigl[ - \be \bR(f)
\bigr]  \pis(df) \Bigr\}
    \]
and
    \[
    \I(\be) = - \log \Bigl\{ \tint \exp \bigl[ - \be \bR(f) 
\bigr]  \pi(df) \Bigr\}.
    \]

\subsection{The localized PAC-Bayes bound} \label{sec:gbound}

With the notation introduced in the previous section, 
we have the following risk bound for any randomized estimator.

\begin{thm} \label{th:gen}
Assume that $\pi$, $\pi^*$, $\cF$ and $\ell$ satisfy the 
integrability conditions \eqref{eq:intega} and \myeq{eq:integb}.
Let $\rho$ be a (posterior) probability distribution on $\cF$ 
admitting a density with respect to $\pi$ depending on $Z_1,\dots,Z_n$.
Let $\hf$ be a prediction function drawn from the distribution $\rho$.
Then for any $\ga \ge 0$, $\gas \ge 0$ and $\eps>0$, with probability (with respect to the distribution $P^{\otimes n} \rho$ generating
the observations $Z_1,\dots,Z_n$ and the randomized prediction function $\hf$) at least $1-\eps$:
	\begin{multline} \label{eq:gen}
	\int \big[ \La(\hf,f) + \gas \bR(f) \big] \pis_{-\gas \bR}(df) - \ga \bR
\bigl(\hf\,\bigr) \\
		\le \I^*(\gas) - \I(\ga)
	    - \log \biggl\{ \int \exp \bigl[ -\cEb(f) \bigr]  \pi(df) \bigg\} 
	    \\ + \log \biggl[  \frac{d \rho}{d \hpi}\bigl(\hf\, \bigr) \biggr]   
+  2\log \bigl(2\eps^{-1} \bigr).
	\end{multline}
\end{thm}

\begin{proof}
See Section \thmref{sec:pgeneric}.
\end{proof}

Some extra work will be needed to prove that Inequality \eqref{eq:gen} provides an upper bound on the excess risk 
$\bR(\hf)$ of the estimator $\hf$. As we will see in the next sections, despite the $-\ga \bR(\hf)$ term and provided that 
$\ga$ is sufficiently small, 
the left-hand side will be essentially lower bounded by $\lam n \bR(\hf)$, while, 
by choosing $\rho=\hpi$, the estimator does not appear in the right-hand side.

\subsection{Application under an exponential moment condition} \label{sec:gfirst}

The estimator proposed in Section \ref{sec:main} and
Theorem \ref{th:gen} seems rather unnatural (or at least complicated) at first sight. 
The goal of this section is twofold. First it shows that 
under exponential moment conditions (i.e., stronger assumptions than the ones in Theorem \ref{th:main} when the linear least square
setting is considered), one can have 
a much simpler estimator than the one consisting in drawing a function according to the distribution \eqref{eq:hpils} 
with $\hcE$ given by \eqref{eq:hcEls} and yet still obtain a $d/n$ convergence rate.
Secondly it illustrates Theorem~\ref{th:gen} in a different and simpler way than the one we will use 
to prove Theorem \ref{th:main}. 

In this section, we consider the following variance and complexity assumptions.

 
{\bf Condition V1.} \label{cond:v1}
There exist $\lam>0$ and $0<\eta<1$ such that for any function $f\in\cF$, 
we have
$\E \Bigl\{ \exp \bigl\{ \lam \, \ela \bigl[Y,f(X)\bigl] \bigr\} \Bigr\}  < +\infty$,\\[-3ex] 
\begin{multline*}
	\log \Bigl\{ \E \Bigl\{ \exp \Bigl\{ 
\lam \, \Bigl[ \ela \bigl[ Y,f(X) \bigr] -  \ela \bigl[ Y,f^*(X) \bigr] 
\Bigr] \Bigr\} \Bigr\} \Bigr\}   \\ \le \lam(1+\eta) [R(f) - R(f^*)],
\end{multline*}
\vspace{-5ex}
\begin{multline*}
\text{and } \log \Bigl\{ \E \Bigl\{ \exp \Bigl\{ 
- \lam \Bigl[ \ela \bigl[Y,f(X)\bigr] - 
\ela \bigl[ Y,f^*(X) \bigr] \Bigr] \Bigr\} \Bigr\} \Bigr\} \\ \le -\lam(1-\eta) [R(f) - R(f^*)].
\end{multline*}

{\bf Condition C.} \label{cond:c}
There exist a probability distribution $\pi$, and constants $D>0$ and $G>0$ such that for any $0<\alpha<\be$,
	\[
	\log \bigg( \frac{\int \exp  \{ 
-\alpha[R(f)-R(f^*)] \} \pi(df)}{\int \exp \{ -\be[R(f)-R(f^*)] \} \pi(df)} \bigg)
		\le D \log\bigg(\frac{G \be}{\alpha}\bigg).
	\]

\begin{thm} \label{th:v1c}
Assume that \textnormal{V1} and \textnormal{C} are satisfied.
Let $\hpig$ be the probability distribution on $\cF$ defined by
its density
	\[
	\frac{d \hpig}{d \pi} (f)  
=\frac{\exp \{ -\lam \sum_{i=1}^n \ela[Y_i,f(X_i)]\}}
		{\int \exp \{ -\lam \sum_{i=1}^n \ela[Y_i,f'(X_i)]\} \pi(df')},
	\]
where $\lam>0$ and the distribution $\pi$ are 
those appearing respectively in \textnormal{V1}
and \textnormal{C}.
Let $\hf \in \cF$ be a function drawn according to this Gibbs distribution.
Then for any $\eta'$ such that $0 < \eta' <1-\eta$ (where $\eta$ is 
the constant appearing in \textnormal{V1}) 
and any $\eps>0$, with probability at least $1-\eps$, we have
	\[
	R( \hf ) - R(f^*) \le \frac{C'_1 D + C'_2 \log(2\eps^{-1})}{n} 
	\]
with 
	\[
	C'_1= \frac{\ds \log \biggl(\frac{G(1+\eta)}{\eta' } \biggr) }{\lam(1-\eta-\eta' )} \quad \text{and} \quad
	C'_2=\frac{2}{\lam(1-\eta-\eta' )}.
	\]
\end{thm}

\begin{proof}
We consider $\ell\bigl[ (X,Y),f,f' \bigr] = \lam \bigl\{ 
\ela\bigl[ Y,f(X) \bigr] -\ela \bigl[ Y,f'(X) \bigr] \bigr\}$, 
where $\lam$ is the constant appearing in the variance assumption.
Let us take $\gas=0$ and let $\pis$ be the Dirac distribution at $f^*$: $\pis(\{f^*\})=1$.
Then Condition V1 implies Condition~I (page \pageref{cond:i}) and 
we can apply Theorem \ref{th:gen}. We have
	\begin{align*}
	L(f,f') & = \lam  \Bigl\{ \ela \bigl[ Y,f(X) \bigr] 
-\ela \bigl[ Y,f'(X) \bigr] \Bigr\},\\
	\hcE(f) & = \lam \sum_{i=1}^n \ela \bigl[ Y_i,f(X_i) \bigr]  
- \lam \sum_{i=1}^n \ela \bigl[ Y_i,f^*(X_i) \bigr], \\
	\hpi & = \hpig,\\
	\La(f) & = - n \log \Bigl\{ \E \Bigr[  \exp \bigl[ -L(f,f^*) 
\bigr] \Bigr] \Bigr\}, \\
	\cEb(f) & = n \log \Bigl\{ \E \Bigl[ \exp \bigl[ L(f,f^*) 
\bigr] \Bigr] \Bigr\}
	\end{align*}
and Assumption V1 leads to: 
\begin{align*}
\log \Bigl\{ \E  \Bigl[ \exp \bigl[ L(f,f^*) \bigr] 
\Bigr] \Bigr\} & \le \lam(1+\eta) [R(f) - R(f^*)]\\
\text{and } 
\log \Bigl\{ \E \Bigl[ \exp \bigl[ -L(f,f^*) \bigr] 
\Bigr] \Bigr\} & \le -\lam(1-\eta) [R(f) - R(f^*)].
\end{align*}
Thus choosing $\rho = \hpi$, \eqref{eq:gen} gives 
	\[
	[\lam n (1-\eta)-\ga] \bR(\hf) 
	    \le - \I(\ga) + \I\bigl[ \lam n (1+\eta) \bigr] + 2\log(2\eps^{-1}).
	\]
Accordingly by the complexity assumption, for $\ga \le \lam n (1+\eta)$, we get
	\[
	[\lam n (1-\eta)-\ga] \bR(\hf) 
		\le D \log\bigg( \frac{G\lam n(1+\eta)}{\ga} \bigg)
		+ 2\log(2\eps^{-1}),
	\]
which implies the announced result by reparameterization (taking $\gamma=\lam n \eta'$).
\end{proof}

Let us conclude this section by mentioning settings in which assumptions V1 and C are satisfied.

\begin{lemma} \label{le:complexity}
Let $\cC$ be a bounded convex set of $\R^d$, and $\vp_1,\dots,\vp_d$ be $d$ square integrable prediction functions. 
Assume that
	$$
	\cF= \big\{ f_\theta=\sum_{j=1}^d \th_j \vp_j ; (\th_1,\dots,\th_d) \in \cC \big\},
	$$
$\pi$ is the uniform distribution on $\cF$ (i.e., the one coming from the uniform distribution on $\cC$),
and that there exist $0<b_1\le b_2$ such that for any $y\in\R$, the function $\ela_y: y' \mapsto \ela(y,y')$
admits almost everywhere a second derivative such that, $(y,y') 
\mapsto \ela_y''(y')$ is measurable, for any $y,y'\in\R$, $b_1 \le \ela''_y(y') \le b_2$, and 
$$
\ela(y,y') = \ela(y,y) + (y'-y) \ela_y'(y) + \int_{y}^{y'} (y' - y'') 
\ela_y''(y'') d y''.
$$
Then Condition \textnormal{C} holds for the above uniform $\pi$, $G=\sqrt{b_2/b_1}$ and $D=d$.

Besides when $f^*=\flin$ (i.e., $\min_{\cF} R = \min_{\th\in\R^d} R(f_\th)$), Condition \textnormal{C} holds for the above uniform $\pi$, $G=b_2/b_1$ and $D=d/2$.
\end{lemma}

\begin{proof}
See Section \thmref{sec:complexity}.
\end{proof}

\begin{rmk} \label{rem:cls}
In particular, for the least squares loss $\ela(y,y')=(y-y')^2$, we have $b_1=b_2=2$ so 
that condition C holds with $\pi$ the uniform distribution on $\cF$, $D=d$ and $G=1$, and with $D=d/2$ and $G=1$ when $f^*=\flin$.
\end{rmk}


\begin{lemma} \label{le:v1b}
Assume that the loss function $\ela$ satisfies the conditions stated 
in Lemma \ref{le:complexity}. Assume moreover that there exist 
$A>0$ and $M>0$ such that for any  $x\in\X$, 
$$
\E \Bigl\{ 
 \exp \Bigl[ A^{-1} \bigl\lvert \ela'_Y \bigl[f^*(X) \bigr] 
\bigr\rvert \Bigr]\, \Big| \,X=x \Bigr\} \le M.
$$
Assume that $\cF$ is convex and has a diameter upper bounded by $H$ for the $L^\infty$-norm:
	\[
	\sup_{f_1,f_2\in \cF,x\in\X} |f_1(x)-f_2(x)| \le H.
	\]
In this case Condition \textnormal{V1} holds for any $(\lam,\eta)$ such that
    \[
    \eta \ge \frac{\lam A^2}{2 b_1} \exp \Bigl[ M^2 \exp \bigl( H b_2/A 
\bigr)  \Bigr] .
    \]
and $0< \lam \le (2AH)^{-1}$ is small enough to ensure $\eta<1$.
\end{lemma}

\begin{proof}
See Section \thmref{sec:pv1b}.
\end{proof}

\subsection{Application without exponential moment condition} \label{sec:gsecond}

When we do not have finite exponential moments as assumed by 
Condition V1 (page \pageref{cond:v1}), e.g., when
	$\E \bigl\{ \exp \bigl\{ \lam \bigl\{ \ela[Y,f(X)]-\ela[Y,f^*(X)] 
\bigr\} \bigr\} \bigr\}  = +\infty$
for any $\lam>0$ and some function $f$ in $\cF$,
we cannot apply Theorem \ref{th:gen} with 
	$\ell\bigl[ (X,Y),f,f' \bigr]  = \lam  \bigl\{ \ela\bigl[Y,f(X)\bigr]-
\ela\bigl[ Y,f'(X) \bigr] \bigr\}$
(because of the $\cEb$ term). However,
we can apply it to the soft truncated excess loss
	\[
	\ell\bigl[ (X,Y),f,f' \bigr] =T
\Bigl( \lam \bigl\{ \ela\bigl[ Y,f(X) \bigr]  - \ela 
\bigl[ Y,f'(X) \bigr] \bigr\} \Bigr),
	\]
with $T(x)=-\log(1-x+x^2/2).$ 
This section provides a result similar to Theorem \ref{th:v1c} in which condition V1 is replaced by the following condition.

{\bf Condition V2.} \label{cond:v2}
For any function $f$, the random variable $\ela \bigl[ Y,f(X) \bigr] 
 - \ela \bigl[ Y,f^*(X) \bigr]$ is square integrable and 
there exists $V>0$ such that for any function $f$, 
	\[
	\E \Bigl\{ \Bigl[ \ela \bigl[ Y,f(X) \bigr]  
- \ela \bigl[ Y,f^*(X) \bigr] \Bigr]^2 \Bigr\} \le V [R(f) - R(f^*)].
	\]

\begin{thm} \label{th:v2c}
Assume that Conditions \textnormal{V2} above and \textnormal{C} 
(page \pageref{cond:c}) are satisfied.
Let $0 < \lam < V^{-1}$ and 
	\beglab{eq:elltrunc}
	\ell \bigl[ (X,Y),f,f' \bigr] =T \Bigl( \lam \big\{ \ela \bigl[Y,f(X)\bigr] 
- \ela \bigl[Y,
f'(X) \bigr] \bigr\}\Bigr),
	\endlab
with 
    \beglab{eq:deft}
    T(x)=-\log(1-x+x^2/2).
    \endlab
Let $\hf \in \cF$ be a function drawn according to the distribution $\hpi$ defined in \myeq{eq:hpi}
with $\hcE$ defined in \myeq{eq:alih} and $\pis=\pi$ the distribution appearing in Condition \textnormal{C}.
Then for any $0 < \eta' <1 - \lam V$ and $\eps>0$, with probability at least $1-\eps$, we have
	\[
	R( \hf ) - R(f^*) \le V \frac{C'_1 D + C'_2 \log(2\eps^{-1})}{n} 
	\]
with 
	\[
	C'_1= \frac{\ds \log \biggl(\frac{G(1+\eta)^2}{\eta'(1-\eta)} \biggr) }{\eta (1-\eta-\eta')}, \quad 
	C'_2=\frac{2}{\eta(1-\eta- \eta')} \quad \text{and} \quad \eta= \lam V.
	\]
In particular, for $\lam=0.32 V^{-1}$ and $\eta'=0.18$, we get
	\[
	R( \hf ) - R(f^*) \le V \frac{16.6 D + 12.5 \log(2\sqrt G\eps^{-1})}{n}.
	\]
\end{thm}

\begin{proof}
We apply Theorem \ref{th:gen} for 
$\ell$ given by \eqref{eq:elltrunc} and $\pis=\pi$.
Let us define, for any $f, f' \in \cF$,  $W(f,f') = \lam \Bigl\{ 
\ela \bigl[ Y,f(X) \bigr] -\ela \bigl[ Y,f'(X)
\bigr]  \Bigr\}$.
Since $\log u \le u-1$ for any $u>0$, we have
	\[
	\La = - n \log \E \bigl( 1 - W + W^2/2  \bigr) \ge n \bigl( 
\E (W) - \E( W^2)/2 \bigr).
	\]
Moreover, from Assumption V2, 
	\beglab{eq:usev}
	\frac{\E \bigl[W(f,f')^2 \bigr]}2\le \E \bigl[ W(f,f^*)^2 \bigr] +
\E \bigl[ W(f',f^*)^2 \bigr] 
		\le \lam^2 V \bR(f)+ \lam^2 V \bR(f'),
	\endlab
hence, by introducing $\eta= \lam V$,
	\begin{align} \label{eq:bla}
	\La(f,f') & \ge \lam n \Bigl[ \bR(f) - \bR(f') - \lam V \bR(f) - \lam V \bR(f') \Bigr] \notag\\
		& = \lam n \Bigl[ (1-\eta) \bR(f) - (1+\eta) \bR(f') \Bigr].
	\end{align}
Noting that 
	\[
	\exp \bigl[ T(u) \bigr] = \frac{1}{1-u+u^2/2}
	\\ = \frac{1 + u + \frac{u^2}{2}}{\bigl( 1 + \tfrac{u^2}{2}
\bigr)^2 - u^2} = \frac{1 + u + \frac{u^2}{2}}{1 + \frac{u^4}{4}}
\leq 1 + u + \frac{u^2}{2}, 
	\]
we see that 
	\[
	\Lb = n \log \Bigl\{ \E \Bigr[  \exp \bigl[ T(W) \bigr] \Bigr] \Bigr\} \le n 
\Bigl[ \E\bigl(W\bigr) + \E \bigl(W^2\bigr)/2  \Bigr].
	\]
Using \eqref{eq:usev} and still $\eta= \lam V$, we get
	\bmul
	\Lb(f,f') \le \lam n \Bigl[ \bR(f) - \bR(f') + \eta \bR(f) + \eta \bR(f') \Bigr]\\
		= \lam n (1+\eta) \bR(f) - \lam n (1-\eta) \bR(f'),
	\end{multline*}
and	
	\beglab{eq:bceb}
	\cEb(f) \le \lam n (1+\eta) \bR(f) - \I \bigl(\lam n(1-\eta) \bigr).
	\endlab
Plugging \eqref{eq:bla} and \eqref{eq:bceb} in \eqref{eq:gen} for $\rho=\hpi$, we obtain
	\bmul
	\bigl[ \lam n (1-\eta) - \ga  \bigr] \bR(\hf) + \bigl[\gas-\lam n (1+\eta) \bigr] \int \bR(f) \pi_{-\gas \bR}(df) \\
	    \le \I(\gas) - \I(\ga) + \I \bigl(\lam n (1+\eta) \bigr) - \I 
\bigl(\lam n (1-\eta) \bigr)+2\log \bigl(2\eps^{-1} \bigr).
	\end{multline*}
By the complexity assumption, choosing $\gas = \lam n (1+\eta)$ and $\ga < \lam n (1-\eta)$, we get
	\[
	 \bigl[ \lam n (1-\eta) - \ga  \bigr] \bR(\hf) 
	    \le D \log\bigg( G \frac{\lam n (1+\eta)^2}{\ga (1-\eta)} \bigg)+2\log 
\bigl(2\eps^{-1} \bigr),
	\]
hence the desired result by considering $\ga= \lam n \eta'$ with $\eta'< 1-\eta$.
\end{proof}

\begin{rmk} \label{rem:complicate}
The estimator seems abnormally complicated at first sight.
This remark aims at explaining why we were not able to consider a simpler estimator.

In Section \ref{sec:gfirst}, in which we consider the exponential moment
condition V1, we took $\ell \bigl[ (X,Y),f,f' \bigr] =
\lam\bigl\{ \ela \bigl[ Y,f(X) \bigr]  - \ela \bigl[ Y,f'(X) \bigr] 
\bigr\}$
and $\pis$ as the Dirac distribution at $f^*$. For these choices, one can easily check
that $\hpi$ does not depend on~$f^*$.

In the absence of an exponential moment condition, we cannot
consider the function $\ell \bigl[ (X,Y),f,f' \bigr] =\lam 
\bigl\{ \ela \bigl[ Y,f(X) \bigr]  - \ela \bigl[ Y,f'(X) \bigr] \bigr\} $
but have instead to use a truncated version. The truncation function $T$ 
of Theorem \ref{th:v2c}
can be replaced by the simpler function $u \mapsto (u\vee -M)\wedge M$
for some appropriate constant $M>0$ but this leads to a bound with worse constants,
without really simplifying the algorithm.
The precise choice $T(x)=-\log(1-x+x^2/2)$ comes from the remarkable property:
there exist second order polynomials $\Pa$ and $\Pb$ such that
	$\frac{1}{\Pa(u)} \le  \exp \bigl[ T(u) \bigr]  \le \Pb(u)$
and
	$\Pa(u)\Pb(u) \le 1+\text{O}(u^4)$ for $u\ra 0$, which 
are reasonable properties to ask in order to ensure 
that \eqref{eq:interv}, and consequently \eqref{eq:gen}, are tight.

Besides, if we take $\ell$ as in \eqref{eq:elltrunc} with $T$ a truncation function 
and $\pis$ as the Dirac distribution at $f^*$,
then $\hpi$ would depend on $f^*$, and is consequently not observable.
This is the reason why we do not consider $\pis$ as the Dirac distribution at $f^*$,
but $\pis=\pi$. This leads to the estimator considered in Theorems \ref{th:v2c} and \ref{th:main}.
\end{rmk}

\begin{rmk} 
Theorem \ref{th:v2c} still holds for the same randomized estimator in which 
\myeq{eq:deft} is replaced with
	\[
	T(x) = \log(1+x+x^2/2).
	\]
\end{rmk}

Condition V2 holds under weak assumptions
as illustrated by the following lemma.

\begin{lemma} \label{le:v2}
Consider the least squares setting: $\ela(y,y') = (y-y')^2$.
Assume that $\cF$ is convex and has a diameter upper bounded by $H$ for the $L^\infty$-norm:
	\[
	\sup_{f_1,f_2\in \cF,x\in\X} |f_1(x)-f_2(x)| \le H
	\]
and that for some $\sigma>0$, we have
	\beglab{eq:moment}
	\sup_{x\in\X} \E\big\{[Y-f^*(X)]^2 \big| X=x\big\} \le \sigma^2 < +\infty.
	\endlab
Then Condition \textnormal{V2} holds for $V=(2\sigma+H)^2$.
\end{lemma}

\begin{proof}
See Section \thmref{sec:pv2}.
\end{proof}

\begin{lemma} \label{le:v3}
Consider the least squares setting: $\ela(y,y') = (y-y')^2$.
Assume that $\cF$ (i.e., $\Theta$) is bounded, and that for any $j\in\{1,\dots,d\}$,
$\E\big[\vp_j(X)^4\big]~<~+~\infty$
and
$\E\big\{\vp_j(X)^2[Y-f^*(X)]^2\big\}<+\infty$.  
Then Condition \textnormal{V2} holds for 
  \begin{align*}
  V= \bigg[ 2& \sqrt{\sup_{f\in\Flin: \E[f(X)^2]=1}\E\big(f(X)^2[Y-f^*(X)]^2\big)}\\
    & \qquad + \sqrt{\sup_{f',f''\in\cF} \E\big([f'(X)-f''(X)]^2\big)} \sqrt{\sup_{f\in\Flin: \E[f(X)^2]=1}\E\big[f(X)^4\big]} \bigg]^2.
  \end{align*}
\end{lemma}

\begin{proof}
See Section \thmref{sec:pv3}.
\end{proof}

\section{Proofs}

\subsection{Main ideas of the proofs}

The goal of this section is to explain the key ingredients appearing in the proofs which both allow to obtain sub-exponential tails 
for the excess risk under a non-exponential moment assumption and get rid of the logarithmic factor in the excess risk bound.

\subsubsection{Sub-exponential tails under a non-exponential moment assumption via truncation}

Let us start with the idea allowing us to prove exponential inequalities under just a moment assumption (instead of the traditional exponential moment assumption).
To understand it, we can consider the (apparently) simplistic $1$-dimensional situation in which we have $\Theta=\R$ and the marginal distribution of $\varphi_1(X)$ is the Dirac distribution at $1$. 
In this case, the risk of the prediction function $f_\th$ is $R(f_\th)=\E 
\bigl[ (Y-\th)^2 \bigr] =\E \bigl[ ( Y- \E Y )^2 \bigr]  
+ (\E Y -\th )^2,$ 
so that the least squares regression problem boils down to the estimation of the mean of the output variable.
If we only assume that $Y$ admits a finite second moment, say $\E (Y^2)\le 1$, it is not clear whether for any $\eps>0$, it is possible to find $\hth$ such that
with probability at least $1-2\eps$,
  \beglab{eq:tar1}
  R(f_{\hth})-R(f^*) = \bigl( \E (Y) - \hth \, \bigr)^2 \le \frac{c \, \logeps}{n},
  \endlab
for some numerical constant $c$.
Indeed, from Chebyshev's inequality, the trivial choice $\hth=\frac{1}{n} 
\sum_{i=1}^n Y_i$ just satisfies: with probability at least $1-2\eps$,
  $$R(f_{\hth})-R(f^*) \le \frac1{n\eps},$$
which is far from the objective \eqref{eq:tar1} for small confidence levels (consider $\eps=\exp(-\sqrt{n})$ for instance).
The key idea is thus to average (soft) \emph{truncated} values of the outputs. This is performed by
taking
  $$\hth = \frac1{n\lam}\sum_{i=1}^n \log\bigg(1+\lam Y_i+\frac{\lam^2Y_i^2}{2}\bigg),$$
with $\lam=\sqrt{\frac{2 \logeps}n}$ (this mean estimator thus depends on the confidence level parameter $\eps$).
Since we have
  $$\log\E \exp(n\lam \hth) = n \log\bigg( 1+ \lam \E(Y) + \frac{\lam^2}2 \E( Y^2 ) \bigg) \le n\lam \E( Y )+ n \frac{\lam^2}2,$$
the exponential Chebyshev's inequality (see Lemma \ref{le:pac}) guarantees that with probability at least $1-\eps$, we have
  $
  n\lam (\hth-\E(Y)) \le n \frac{\lam^2}2 + \logeps
  $,
hence
  $$\hth-\E (Y)\le \sqrt\frac{2\logeps}{n}.$$
Replacing $Y$ by $-Y$ in the previous argument, we obtain that 
with probability at least $1-\eps$, we have
  $$
  n\lam \bigg\{ \E(Y)+\frac1{n\lam}\sum_{i=1}^n \log\bigg(1-\lam Y_i+\frac{\lam^2Y_i^2}{2}\bigg) \bigg\} \le n \frac{\lam^2}2 + \logeps.
  $$

Since $-\log(1+x+x^2/2) \le \log(1-x+x^2/2)$, this implies
$$
\E (Y)-\hth\le \sqrt\frac{2\logeps}{n}.
$$
The two previous inequalities imply Inequality \eqref{eq:tar1} (for $c={2}$), showing that sub-exponential tails are achievable even when we only assume that the random variable admits a finite second moment (see \cite{Cat09} for more details 
on the robust estimation of the mean of a random variable).

\subsubsection{Localized PAC-Bayesian inequalities to eliminate a logarithm factor}


The analysis of statistical inference generally relies on upper bounding the supremum of an empirical process $\chi$ indexed by the functions in a model $\cF$. One central tool to obtain these bounds are the concentration inequalities.
An alternative approach, called the PAC-Bayesian one, consists in using the entropic equality
  \beglab{eq:iexp}
  \E \exp\Bigg( \sup_{\rho\in\M} \bigg\{ \int \rho(df) \chi(f) - K(\rho,\pi') \bigg\} \Bigg)= \int\pi'(df) \E \exp\big( \chi(f) \big).
  \endlab
where $\M$ is the set of probability distributions on $\cF$ and $K(\rho,\pi')$ is the Kullback-Leibler divergence (whose definition is recalled in \myeq{eq:kl}) between $\rho$ and some fixed distribution $\pi'$.

Let $\chr:\cF \ra \R$ be an observable process such that for any $f\in\cF$, we have
  $$
  \E \exp\big(\chi(f)\big) \le 1
  $$
for $\chi(f) =\lam[ R(f) - \chr(f)]$ and some $\lam>0$. Then, as a consequence
of  \eqref{eq:iexp}, for any $\eps>0$, with probability at least $1-\eps$,
for any distribution $\rho$ on $\cF$, 
  \beglab{eq:ipac}
  \int \rho(df) R(f) \le \int \rho(df) \chr(f) + \frac{K(\rho,\pi')+\logeps}{\lam}.
  \endlab
The left-hand side quantity represents the expected risk with respect to the distribution $\rho$. 
The question is now how to use \eqref{eq:ipac} to design a posterior distribution $\rho$ for which 
$\int \rho(df) R(f)$ is guaranteed to be small.
The constraint on the choice of $(\rho,\pi')$ is that $\rho$ should be computable from the data (e.g., it cannot depend on $R$) and $\pi'$ should not depend on the data: it may depend on $R$ (in contrast with Bayesian prior distributions!) but not on $\chr$.
Simple choices like $(\rho,\pi')=(\delta_{f^*},\delta_{f^*})$ or $(\rho,\pi')=(\delta_{\check{f}},\delta_{\check{f}})$ for $\check{f}\in\argmin_{f\in\cF} \chr(f)$, where $\delta_a$ denotes the Dirac distribution at the function $f$, are thus forbidden (while they would have led to small right-hand side of \eqref{eq:ipac}).

For fixed $\pi'$, the posterior distribution minimizing the right-hand side of \eqref{eq:ipac} is $\rho=\pi'_{-\lam \chr}$. It is computable from the data if $\pi'$ is. 
Without prior knowledge, this would lead to take a ``flat'' distribution 
for $\pi'$ (e.g., the one induced by the Lebesgue measure in the case of a model $\cF$ defined by a bounded parameter set in some Euclidean space). The resulting Kullback-Leibler divergence might be very large
as it compares a distribution with a sharp peak (concentrated on functions $f\in\cF$ for which $\chr(f)$) with a flat one. 

To get a smaller Kullback-Leibler divergence, we can take posterior and prior distributions which are peaked around almost the same function. This can be done by taking $\pi$ and $\rho$ respectively concentrated around $f^*$ and $\check{f}$. More precisely, one can take posterior distributions of the form $\rho=\pi_{-\lam \chr}$ for some $\lam>0$ and a ``flat'' distribution $\pi$ computable without knowing neither the distribution $P$ generating the data nor the training data (in particular, $\pi$ must not depend on $R$ or $\chr$), and a ``localized'' prior distribution $\pi'=\pi_{-\be R}$ for some $\be>0$.
The parameters $\lam$ and $\be$ controlling the sharpness of the peaks at $\argmin_{f\in\cF} R(f)*$ and $\argmin_{f\in\cF} \chr(f)$ 
should be taken such that the peaks overlap (to ensure that the 
Kullback-Leibler divergence is small) and 
are in the same time sharp enough (to ensure that $\int \rho(df) \chr(f)$ is small).
The use of the ``localized'' prior distribution $\pi'=\pi_{-\be R}$ implies an additional technical difficulty as one needs to control the 
divergence $K(\rho,\pi_{-\be R})$. This is achieved by writing
  $$
  K(\rho,\pi_{-\be R})=K(\rho,\pi)+\log\bigg( \int \exp[-\be R(f)]\, \pi(df) \bigg) + \beta \int R(f) \, \rho(df),
  $$ 
and controlling the new logarithmic term through PAC-Bayesian inequalities.

\subsection{Proof of Theorem \ref{th:gen}} \label{sec:pgeneric}

We use the standard way of obtaining PAC bounds through 
upper bounds on Laplace transforms of appropriate random variables.
This argument is synthesized in the following result.

\begin{lemma} \label{le:pac}
For any $\eps>0$ and any real-valued random variable $V$ such that $\E \bigl[ 
\exp(V) \bigr] \le 1$,
with probability at least $1-\eps$, we have
    \[
    V \le \logeps.
    \]
\end{lemma}
\begin{multline*}
\text{Let }  V_1(\hf) = \int \big[ \La(\hf,f) + \gas \bR(f) \big] \pis_{-\gas \bR}(df) - \ga \bR(\hf) \\
		- \I^*(\gas) + \I(\ga)
	    + \log \bigg(\int \exp \bigl[ -\hcE(f) \bigr] \pi(df)\bigg) 
	    - \log \biggl[  \frac{d \rho}{d \hpi}\bigl(\hf\bigr) \biggr],
    \end{multline*}
    \[
\text{and }    V_2 = - \log \bigg(\int \exp\bigl[ -\hcE(f)
\bigr]  \pi(df)\bigg) + \log \bigg(\int \exp \bigl[ -\cEb(f)\bigr] \pi(df)\bigg)
    \]
To prove the theorem, according to Lemma \ref{le:pac}, it suffices to prove that
$$
\E \Bigl\{ \tint \exp \bigl[ V_1(\hf) \bigr]  \rho(d \hf) 
\Bigr\}  \le 1
\quad \text{and} \quad \E \Bigl[ \tint \exp (V_2 ) \rho(d \hf) \Bigr]  \le 1.
$$
These two inequalities are proved in the following two sections.

\subsubsection{Proof of $\E \Bigl\{ \int \exp \bigl[ V_1(\hf) 
\bigr]  \rho(d \hf) \Bigr\} \le 1$}

From Jensen's inequality, we have
    \begin{align*}
    \int & \big[ \La(\hf,f) + \gas \bR(f) \big] \pis_{-\gas \bR}(df)\\
        & = \int \big[ \hL(\hf,f) + \gas \bR(f) \big]  \pis_{-\gas \bR}(df) + \int \big[ \La(\hf,f) - \hL(\hf,f) \big] \pis_{-\gas \bR}(df)\\
        & \le \int \big[ \hL(\hf,f) + \gas \bR(f) \big] \pis_{-\gas \bR}(df) + \log \int \exp \bigl[ \La(\hf,f) - \hL(\hf,f) \bigr] \pis_{-\gas \bR}(df).
    \end{align*}

From Jensen's inequality again, 
    \begin{align*}
    - \hcE(\hf) & = - \log \int \exp \bigl[ \hL(\hf,f) \bigr] \pis(df)\\
    & = - \log \int \exp \bigl[ \hL(\hf,f)+\gas \bR(f) \bigr] \pis_{-\gas \bR}(df)
        - \log \int \exp \bigl[ -\gas \bR(f)\bigr]  \pis(df)\\
    & \le - \int [ \hL(\hf,f)+\gas \bR(f) ] \pis_{-\gas \bR}(df) + \I^*(\gas).
    \end{align*}
    
From the two previous inequalities, we get
    \begin{align*}
    V_1(\hf) & \le \int \big[ \hL(\hf,f) + \gas \bR(f) \big]  \pis_{-\gas \bR}(df) \\ 
& \qquad + \log \int \exp \bigl[ \La(\hf,f)- \hL(\hf,f) \bigr]  
\pis(df) - \ga \bR(\hf) \\
	& \qquad - \I^*(\gas) + \I(\ga)
	    + \log \bigg(\int \exp \bigl[ -\hcE(f) \bigr] \pi(df)\bigg) 
	    - \log \biggl[  \frac{d \rho}{d \hpi}(\hf) \biggr],\\
    & = \int \big[ \hL(\hf,f) + \gas \bR(f) \big] \pis_{-\gas \bR}(df) 
\\ & \qquad + \log \int \exp \bigl[ \La(\hf,f)- \hL(\hf,f)  \bigr]  \pis(df) - \ga \bR(\hf) \\
	& \qquad - \I^*(\gas) + \I(\ga) - \hcE(\hf) - \log\biggl[ 
\frac{d \rho}{d \pi}(\hf) \biggr],\\
    & \le \log \int \exp \bigl[ \La(\hf,f)- \hL(\hf,f) \bigr]  
\pis_{-\gas \bR}(df) \\ & \qquad - \ga \bR(\hf) + \I(\ga) - \log\biggl[ 
\frac{d \rho}{d \pi}(\hf) \biggr]\\
    & = \log \int \exp \bigl[ \La(\hf,f)- \hL(\hf,f) \bigr]  \pis_{-\gas \bR}(df) + \log \biggl[ \frac{d \pi_{-\ga \bR}}{d \rho}(\hf) \biggr],
	\end{align*}
hence, by using Fubini's inequality and the equality 
\begin{multline*}
{} \hfill \E \Bigl\{ \exp \bigl[ - \hL(\hf,f) \bigr] \Bigr\}  = \exp \bigl[ -\La(\hf,f)
\bigr], \hfill {} \\ 
\shoveleft{\hspace{-2ex}\text{we obtain } \E \int \exp \bigl[ V_1(\hf) \bigr]  
\rho(d\hf)}  \\ \le \E \int \bigg(\int \exp \bigl[ \La(\hf,f) - \hL(\hf,f)\bigr]  
\pis_{-\gas \bR}(df) \bigg) \pi_{-\ga \bR}(d\hf)\\
     =  \int \bigg(\int \E \exp \bigl[ \La(\hf,f) - \hL(\hf,f) \bigr] 
\pis_{-\gas \bR}(df) \bigg) \pi_{-\ga \bR}(d\hf)
     = 1. 
	\end{multline*}
	
\subsubsection{Proof of $\,\E \Bigl[ \int \exp (V_2) \rho(d \hf) \Bigr]  \le 1$}

It relies on the following result.

\begin{lemma} \label{le:concpart}
Let $\W$ be a real-valued measurable function defined on a product space $\A_1\times\A_2$ and
let $\mu_1$ and $\mu_2$ be probability distributions on respectively $\A_1$ and $\A_2$.
\bi
\item if $\expec{a_1}{\mu_1} \Bigl\{ 
 \log \Bigl[ \expec{a_2}{\mu_2} \bigl\{ \exp \bigl[ -\W(a_1,a_2) \bigr] 
\bigr\} \Bigr] \Bigr\}  < +\infty$, then we have
\begin{multline*}
        - \expec{a_1}{\mu_1} \Bigl\{ 
\log \Bigl[ \expec{a_2}{\mu_2} \bigl\{ \exp \bigl[ -\W(a_1,a_2)\bigr] \bigr\}
\Bigr] \Bigr\}  \\ 
\le - \log \Bigl\{ \expec{a_2}{\mu_2} \Bigl[ \exp \bigl[ 
-\expec{a_1}{\mu_1} \W(a_1,a_2) \bigr] \Bigr] \Bigr\}.
\end{multline*}
\item if $\W>0$ on $\A_1\times\A_2$ and $\expec{a_2}{\mu_2} 
\Bigl\{ \expec{a_1}{\mu_1} \bigl[ \W(a_1,a_2) \bigr]^{-1} 
\Bigr\}^{-1} < +\infty$, then 
$$
\expec{a_1}{\mu_1} \Bigl\{ \expec{a_2}{\mu_2} \Bigl[ \W(a_1,a_2)^{-1} 
\Bigr]^{-1} \Bigr\}  \le \expec{a_2}{\mu_2} \Bigl\{ \expec{a_1}{\mu_1} 
\bigl[ \W(a_1,a_2) \bigr]^{-1} \Bigr\}^{-1}.
$$

\ei
\end{lemma}
\noindent{\sc Proof.}
\bi
\item 
Let $\A$ be a measurable space and $\M$ denote the set of probability distributions on $\A$.
The Kullback-Leibler divergence between a distribution $\rho$ and a distribution~$\mu$ is
\beglab{eq:kl}
K(\rho,\mu) \eqdef \begin{cases}\ds  
\undc{\E}{a\sim\rho} \log \biggl[ \frac{d \rho}{d \mu}(a) \biggr] 
& \text{if } \rho \ll \mu,\\
\ds + \infty & \text{otherwise,}
\end{cases} 
\endlab
where $\ds \frac{d \rho}{d \mu}$ denotes as usual 
the density of $\rho$ $\wrt$ $\mu$.
The Kullback-Leibler divergence satisfies the duality formula (see, e.g., 
\cite[page 159]{Cat01}): for any 
real-valued measurable function $h$ defined on $\A$,
\beglab{eq:legendre}
\und{\inf}{\rho\in\M} \big\{ \undc{\E}{a\sim\rho} h(a) +
K(\rho,\mu) \big\} = -\log \expec{a}{\mu} \Bigl\{ \exp \bigl[ -h(a) \bigr] \Bigr\}.
\endlab
By using twice \eqref{eq:legendre} and Fubini's theorem, we have
        \begin{align*}
        - \expec{a_1}{\mu_1} \Bigl\{ \log  \Bigl\{ & 
\expec{a_2}{\mu_2} \Bigl[ \exp \bigl[
-\W(a_1,a_2) \bigr]\Bigr] \Bigr\} \Bigr\}\\
                & = \expec{a_1}{\mu_1} \Bigl\{ 
\und{\inf}{\rho} \big\{ \expec{a_2}{\rho}  \bigl[ \W(a_1,a_2) \bigr]  + K(\rho,\mu_2) \big\} \Bigr\} \\
        & \le \und{\inf}{\rho} \Bigl\{ \expec{a_1}{\mu_1} 
\Bigl[ \expec{a_2}{\rho} \bigl[ \W(a_1,a_2) \bigr] + K(\rho,\mu_2) 
\Bigr] \Bigr\} \\
        & = - \log \Bigl\{ \expec{a_2}{\mu_2} 
\Bigl[ \exp \bigl\{ -\expec{a_1}{\mu_1} \bigl[ \W(a_1,a_2) \bigr] 
\bigr\} \Bigr] \Bigr\}.
        \end{align*}
\item By using twice \eqref{eq:legendre} and the first assertion of Lemma \ref{le:concpart}, we have
        \begin{multline*}
\expec{a_1}{\mu_1} \Bigl\{ \expec{a_2}{\mu_2} \Bigl[ \W(a_1,a_2)^{-1} 
\Bigr]^{-1} \Bigr\} 
           \\ = \expec{a_1}{\mu_1}\Bigl\{ 
\exp \Bigl\{ -\log \Bigl[ \expec{a_2}{\mu_2} \bigl\{ 
\exp \bigl[-\log \W(a_1,a_2)
\bigr] \bigr\} \Bigr] \Bigr\} \Bigr\} \\
= \expec{a_1}{\mu_1} \Bigl\{ \exp \Bigl\{ \inf_\rho \Bigl[ 
\expec{a_2}{\rho}  \bigl\{ \log \bigl[ \W(a_1,a_2) \bigr] 
\bigr\}  + K(\rho,\mu_2) \Bigr] \Bigr\} \Bigr\}  \\
           \le \inf_\rho \Bigl\{ 
\exp \bigl[ K(\rho,\mu_2) \bigr]  \expec{a_1}{\mu_1} 
\Bigl\{ \exp \Bigl\{  \expec{a_2}{\rho} 
\Bigl[ \log \bigl[ \W(a_1,a_2) \bigr] \Bigr] \Bigr\} \Bigr\} \\
           \le \inf_\rho \Bigl\{ \exp \bigl[  
K(\rho,\mu_2)\bigr]  \exp \Bigl\{ \expec{a_2}{\rho} 
\Bigl\{ \log \Bigl[ \expec{a_1}{\mu_1} \bigl[ \W(a_1,a_2) \bigr] 
\Bigr] \Bigr\} \Bigr\} \\
= \exp \Bigl\{ \inf_\rho \Bigl\{ 
\expec{a_2}{\rho} \Bigl[ \log \bigl\{ \expec{a_1}{\mu_1} 
\bigl[ \W(a_1,a_2) \bigr] \bigr\} \Bigr]  +K(\rho,\mu_2) \Bigr\} \Bigr\} \\
= \exp \Bigl\{ - \log \Bigl\{ 
\expec{a_2}{\mu_2} \Bigl\{ \exp \Bigl[ - \log 
\bigl\{ \expec{a_1}{\mu_1} \bigl[ \W(a_1,a_2) \bigr] \bigr\} \Bigr] \Bigr\} 
\Bigr\} \Bigr\} \\
= \expec{a_2}{\mu_2} \Bigl\{ \expec{a_1}{\mu_1} \bigl[ \W(a_1,a_2) 
\bigr]^{-1} \Bigr\}^{-1}. \quad \square
\end{multline*}
\ei

From Lemma \ref{le:concpart} and Fubini's theorem, since $V_2$ does not depend on $\hf$, we have
\begin{multline*}
\E \biggl[ \int  \exp (V_2) \rho(d \hf) 
\biggr]    = \E \bigl[ \exp (V_2) \bigr] \\
= \int \exp \bigl[-\cEb(f) \bigr] \,  \pi(df) 
\, \E \biggl\{ \biggl[ \int \exp \bigl[ -\hcE(f) \bigr] \, \pi(df) 
\biggr]^{-1} \biggr\} \\
\le \int \exp \bigl[ -\cEb(f) \bigr]  \pi(df) 
\, \biggl\{  \int \E \bigl[ \exp \bigl( \hcE(f) \bigr) \bigr]^{-1} 
\pi(df) \biggr\}^{-1}\\
= \int \exp \bigl[ -\cEb(f) \bigr] \pi(df) \, \biggl\{ 
\int \E \biggl[ \int \exp \bigl[ \hL(f,f') \bigr] \pis(df') \biggr]^{-1}  
\pi(df) \biggr\}^{-1}\\
     = \int \exp \bigl[ -\cEb(f) \bigr] \pi(df) \biggl\{ 
 \int \biggl[ \int \exp \bigl[ \Lb(f,f') \bigr]  \pis(df') 
\biggr]^{-1}  \pi(df) \biggr\}^{-1}
= 1.
\end{multline*}
This concludes the proof that
for any $\ga \ge 0$, $\gas \ge 0$ and $\eps>0$, with probability 
(with respect to the distribution $P^{\otimes n} \rho$ generating
the observations $Z_1,\dots,Z_n$ and the randomized prediction function $\hf$) 
at least $1-2\eps$:
\[
V_1(\hf)+V_2 \le 2 \logeps.
\]

\subsection{Proof of Lemma \ref{le:complexity}} \label{sec:complexity}

Let us look at $\cF$ from the point of view of $f^*$. Precisely 
let $\S_{\R^d}(O,1)$ be the sphere of $\R^d$ centered at the origin and with radius $1$
and 
	\[
	\S=\biggl\{ \sum_{j=1}^d \th_j \vp_j 
; (\th_1,\dots,\th_d) \in \S_{\R^d}(O,1) \biggr\}.
	\] 
Introduce 
	$$
	\Omega = \big\{ \sigmb \in \S; \exists u>0 \text{ s.t. } f^* + u \sigmb \in \cF \big\}.
	$$
For any $\sigmb \in \Omega$, let $u_\sigmb = \sup\{u>0:f^*+ u \sigmb \in \cF \}$.
Since $\pi$ is the uniform distribution on the convex set $\cF$ (i.e., the one coming from the uniform distribution on $\cC$), 
we have
\begin{multline*}
	\int \exp \bigl\{ -\alpha[R(f)-R(f^*)] 
\bigr\}  \pi(df) \\ = \int_{\sigmb\in\Omega} \int_0^{u_\sigmb} 
\exp \bigl\{ -\alpha [R(f^*+u\sigmb)-R(f^*)] \bigr\}  
u^{d-1} du d\sigmb.
\end{multline*}
Let $c_\sigmb=\E [\sigmb(X) \ela_Y'(f^*(X))]$
and $a_\sigmb = \E \bigl[ \sigmb^2(X) \bigr]$. 
Since 
$$
f^*\in\argmin_{f\in\cF} \E \bigl\{ \ela_Y \bigl[ f(X) 
\bigr] \bigr\}, 
$$
we have $c_\sigmb\ge 0$ (and $c_\sigmb=0$ if both $-\sigmb$ and $\sigmb$ belong to $\Omega$).
Moreover from Taylor's expansion, 
$$
\frac{b_1 a_\sigmb u^2}{2} \le 
	R(f^*+u\sigmb)-R(f^*) - u c_\sigmb 
		\le \frac{ b_2 a_\sigmb u^2}{2}.
$$
Introduce 
	\[
	\psi_\sigmb = \frac{\int_0^{u_\sigmb} \exp \bigl\{ 
-\alpha[ u c_\sigmb +\demi b_1 a_\sigmb u^2]\bigr\} u^{d-1} du} 
{\int_0^{u_\sigmb} \exp \bigl\{ -\be[ u c_\sigmb+\demi b_2 a_\sigmb u^2]
\bigr\}  u^{d-1} du}.
	\]
For any $0<\alpha<\be$, we have
	\[
	 \frac{\int \exp \bigl\{ -\alpha[R(f)-R(f^*)] \bigr\}  \pi(df)}{
\int \exp \bigl\{ -\be[R(f)-R(f^*)]\bigr\}  \pi(df)} 
		\le \und{\inf}{\sigmb\in\S} \psi_\sigmb.
	\]	
For any $\zeta> 1$, by a change of variable, 
	\begin{align*}
	 \psi_\sigmb & < \zeta^d
		\frac{\int_0^{u_\sigmb} \exp \bigl\{ -\alpha[ \zeta u c_\sigmb +\demi b_1 a_\sigmb \zeta^2 u^2] \bigr\} u^{d-1} du}
		{\int_0^{u_\sigmb} \exp \bigl\{ -\be[ u c_\sigmb +\demi b_2 a_\sigmb u^2]
\bigr\} u^{d-1} du}\\
	 & \le \zeta^d \und{\sup}{u>0}
		\exp \bigl\{ \be[ u c_\sigmb + \tfrac{1}{2} b_2 a_\sigmb u^2]
		 -\alpha[ \zeta u c_\sigmb + \tfrac{1}{2} b_1 a_\sigmb \zeta^2 u^2]\bigr\}.
	\end{align*}
Taking $\zeta=\sqrt{(b_2\be)/(b_1\alpha)}$ when $c_\sigmb=0$ and 
$\zeta=\sqrt{(b_2\be)/(b_1\alpha)} \vee (\be/\alpha)$ otherwise,
we obtain $\psi_\sigmb < \zeta^d$, hence
$$
	\log \bigg( \frac{\int \exp \bigl\{ -\alpha[R(f)-R(f^*)]
\bigr\} \pi(df)}{\int \exp \bigl\{ -\be[R(f)-R(f^*)]
\bigr\}  \pi(df)} \bigg)
 \le \begin{cases}
\ds 		\frac{d}{2} \log\big(\frac{b_2\be}{b_1\alpha}\big) 
 \text{ when } \sup_{\sigmb\in\Omega} c_\sigmb=0,\\[2ex] 
\ds d \log\big(\sqrt{\frac{b_2\be}{b_1\alpha}} 
\vee \frac{\be}{\alpha} \big)  \text{ otherwise,}
		\end{cases} 
$$
which proves the announced result.

\subsection{Proof of Lemma \ref{le:v1b}} \label{sec:pv1b}

For $-(2AH)^{-1} \le \lam \le (2AH)^{-1}$, introduce the random variables
    \[
    F = f(X) \quad \text{ \quad } \quad F^* = f^*(X),
    \]
    \[
    \Omega = \ela'_Y(F^*)+(F-F^*) \int_0^1 (1-t) \ela''_Y(F^*+t(F-F^*)) dt,
    \]
    \[
    L=\lam [\ela(Y,F) - \ela(Y,F^*) ],
    \]
and the quantities
    \[
    a(\lam) = \frac{M^2 A^2 \exp (H b_2/A)}{2\sqrt{\pi}(1-|\lam| AH)}
    \]
and 
    \[
    \tA = H b_2/2 + A \log (M) = \frac{A}{2} \log 
\bigl\{  M^2 \exp \bigl[Hb_2/(2A) \bigr] \bigr\}.
    \]

From Taylor-Lagrange formula, we have
    \[
    L = \lam (F-F^*) \Omega.
    \]
    
Since $\E \bigl[  \exp \bigl( |\Omega|/A 
\bigr)\,|\, X \bigr] \le M \exp \bigl[ H b_2/(2A) \bigr]$, 
Lemma \ref{le:stdb} gives
    \[
    \log 
\Bigl\{ \E \Bigr[ \exp \bigl\{ 
\alpha [\Omega - \E(\Omega | X )]/A \bigr\}\, | \, X \Bigr] \Bigr\}  
\le \frac{M^2 \alpha^2 \exp \bigl(H b_2/A\bigr)}{2\sqrt{\pi} (1-|\alpha|)}
    \]
for any ${-1}<\alpha<1$,
and
    \beglab{eq:upbomega}
    \big| \E(\Omega|X) \big| \le \tA.
    \endlab
By considering $\alpha=A\lam[f(x)-f^*(x)]\in[-1/2;1/2]$ for fixed $x\in\X$, we get
    \beglab{eq:tm1}
    \log \Bigl\{ \E \Bigl[ 
\exp \bigl[ L-\E(L|X) \bigr]\,|\, X \Bigr] \Bigr\} \le \lam^2 (F-F^*)^2 a(\lam).
    \endlab
Let us put moreover 
    \[
    \tL= \E(L|X) + a(\lam) \lam^2 (F-F^*)^2.
    \]
Since $-(2AH)^{-1} \le \lam \le (2AH)^{-1}$, we have
    $
    \tL \le |\lam| H \tA + a(\lam) \lam^2 H^2 \le b'
    $
with $b'=\tA/(2A) + M^2 \exp \bigl( H b_2/A \bigr) /(4\sqrt{\pi})$.
Since $L-\E(L) = L-\E( L|X) + \E( L|X) - \E (L)$, by using Lemma \ref{le:stda},
\eqref{eq:tm1} and \eqref{eq:upbomega}, we obtain
    \begin{align*}
    \log \Bigl\{ \E \Bigl[ \exp \bigl[L-\E(L) \bigr] \Bigr] \Bigr\} 
        & \le \log \Bigl\{ \E \Bigl[ \exp \bigl[ \tL-\E 
(\tL) \bigr] \Bigr] \Bigr\} + \lam^2 a(\lam) \E 
\bigl[ (F-F^*)^2 \bigr] \\
    & \le \E \bigl( \tL^2 \bigr) g( b' ) + \lam^2 a(\lam) \E 
\bigl[ (F-F^*)^2 \bigr] \\
    & \le \lam^2 \E \bigl[ (F-F^*)^2 \bigr] \big[ {\tA}^2 g( b' ) + a(\lam) \big],
    \end{align*}       
with $g(u)= \bigl[ \exp(u) - 1 - u \bigr] /u^2$.
Computations show that for any $-(2AH)^{-1} \le  \lam \le (2AH)^{-1}$,
    \[
    \tA^2 g( b' ) + a(\lam) \le \frac{A^2}{4} \exp \Bigl[ M^2 \exp \bigl(H b_2/A \bigr) \Bigr].
    \]
Consequently, for any $-(2AH)^{-1} \le  \lam \le (2AH)^{-1}$, we have 
\begin{multline*}
    \log 
\Bigl\{ \E \Bigl[  \exp \bigl\{ \lam [\ela(Y,F) - \ela(Y,F^*) ]
\bigr\} \Bigr] \Bigr\} \\ \le \lam[R(f)-R(f^*)] 
    + \lam^2 \E \bigl[ (F-F^*)^2 \bigr] \frac{A^2}{4}  
\exp \Bigl[ M^2 \exp \bigl(H b_2/A \bigr) \Bigr].
\end{multline*}
Now it remains to notice that $\E \bigl[ (F-F^*)^2  
\bigr] \le 2 [R(f) - R(f^*)]/b_1.$
Indeed consider the function $\phi(t) = R(f^*+t(f-f^*))-R(f^*),$
where $f\in\cF$ and $t\in[0;1]$. From the definition of $f^*$
and the convexity of $\cF$, we have $\phi\ge 0$ on $[0;1]$, implying 
that $\phi'(0) \geq 0$. Besides
	$\phi(1)=\phi(0)+\phi'(0)+ \int_0^1 (1 - t) \phi''(t) dt$,
where $\phi''(t)$ is defined as
\begin{align*}
\phi''(t) & = \E \Bigl\{ \bigl[ f(X) - f^*(X) \bigr]^2 \ela_Y'' \bigl[ [(1-t)f^*
+ f](X)\bigr] \Bigr\}\\
& \geq b_1 \E \bigl\{ \bigl[ f(X) - f^*(X) \bigr]^2 \bigr\},
\end{align*}
implying that 
	\beglab{eq:proj}
	\frac{b_1}{2} \E(F-F^*)^2 \le R(f) - R(f^*).
	\endlab

\subsection{Proof of Lemma \ref{le:v2}} \label{sec:pv2}

We have
    \begin{align*}
    & \E\Big( \big\{ [Y-f(X)]^2-[Y-f^*(X)]^2 \big\}^2 \Big)\\
       = \ & \E\Big( [f^*(X)-f(X)]^2\big\{2[Y-f^*(X)]+[f^*(X)-f(X)]\big\}^2 \Big)\\
    = \ & \E\Big( [f^*(X)-f(X)]^2\big\{4\E\big([Y-f^*(X)]^2\big|X\big) \\
    & \qquad \qquad +4\E(Y-f^*(X)|X)[f^*(X)-f(X)]+[f^*(X)-f(X)]^2\big\} \Big)\\
    \le \ & \E\Big( [f^*(X)-f(X)]^2\big\{4\sigma^2 +4\sigma |f^*(X)-f(X)|+[f^*(X)-f(X)]^2\big\} \Big)\\
    \le \ & \E\Big( [f^*(X)-f(X)]^2(2\sigma+H)^2\Big)\\
    \le \ & (2\sigma+H)^2 [R(f)-R(f^*)],
    \end{align*}
where the last inequality is the usual relation between excess risk and $L^2$ distance 
using the convexity of $\cF$ (see above \eqref{eq:proj} for a proof).

\subsection{Proof of Lemma \ref{le:v3}} \label{sec:pv3}

Let $\cS = \{ s\in\Flin: \E[s(X)^2]=1\}$. 
Using the triangular inequality in $\B{L}^2$, we get
    \begin{align*}
    & \E\Big( \big\{ [Y-f(X)]^2-[Y-f^*(X)]^2 \big\}^2 \Big)\\
       = \ & \E\Big( \big\{2[f^*(X)-f(X)][Y-f^*(X)]+[f^*(X)-f(X)]^2\big\}^2 \Big)\\
     \le \ & \Big( 2\sqrt{\E\big\{[f^*(X)-f(X)]^2[Y-f^*(X)]^2\big\}}+\sqrt{\E\big\{[f^*(X)-f(X)]^4\big\}} \Big)^2\\
     \le \ & \bigg[ 2\sqrt{\E\big([f^*(X)-f(X)]^2\big)}\sqrt{\sup_{s\in\cS}\E\big(s(X)^2[Y-f^*(X)]^2\big)}\\
    & \qquad + \E\big([f^*(X)-f(X)]^2\big) \sqrt{\sup_{s\in\cS}\E\big[s(X)^4\big]} \bigg]^2\\
    \le \ & V [R(f)-R(f^*)],
    \end{align*}
with 
  \begin{align*}
  V= \bigg[ 2& \sqrt{\sup_{s\in\cS}\E\big(s(X)^2[Y-f^*(X)]^2\big)}\\
    & \qquad + \sqrt{\sup_{f',f''\in\cF} \E\big([f'(X)-f''(X)]^2\big)} \sqrt{\sup_{s\in\cS}\E\big[s(X)^4\big]} \bigg]^2,   
  \end{align*}
where the last inequality is the usual relation between excess risk and $L^2$ distance 
using the convexity of $\cF$ (see above \eqref{eq:proj} for a proof).

\appendix

\section{Uniformly bounded conditional variance is necessary to reach $d/n$ rate} \label{sec:lb}

In this section, we show that the target \eqref{eq:exptarget} cannot be 
reached if we just assume that $Y$ has a finite variance and that the functions 
in $\cF$ are bounded. For this purpose, the following result gives a $1/\sqrt{n}$ lower bound when $d=2$. (Note that it is not implied by the 
$\sqrt{\fracc{\log(1+ d/\sqrt{n})}n}$ lower bound for convex aggregation, proved in
\cite{Tsy03}, and in slightly weaker forms in \cite{Jud00,Yan01b}, since the latter bound is shown for $d\ge \sqrt{n}$.)

For this, consider an input space $\X$ partitioned into two sets $\X_1$ and
$\X_2$: $\X=\X_1\cup\X_2$ and $\X_1\cap\X_2=\emptyset$.
Let $\vp_1(x)=\dsone_{x\in\X_1}$ and $\vp_2(x)=\dsone_{x\in\X_2}$.
Let $\cF= \big\{ \th_1 \vp_1+\th_2 \vp_2 ; (\th_1,\th_2) \in [-1,1]^2 \big\}.$
\begin{thm} \label{th:lb}
For any estimator $\hf$ and any training set size $n\ge 1$, we have
	\beglab{eq:lb}
	\und{\sup}{P} \big\{ \E \bigl[ R(\hf \,) \bigr] - R(f^*) \big\}
		\ge \frac{1}{4\sqrt{n}},
	\endlab
where the supremum is taken with respect to all probability distributions
such that $\freg \in \cF$ and $\Var (Y) \le 1$.
\end{thm}

\begin{proof}
Let $\be$ satisfying $0<\be\le 1$ be some parameter to be chosen later.
Let $P_\sigma$, $\sigma \in\{-,+\}$, be two probability distributions on
$\X \times \R$ such that for any $\sigma \in\{-,+\}$,
	\[P_\sigma(\X_1) = 1-\be,\]
	\[P_\sigma(Y=0|X=x) = 1 \qquad \text{for any } x\in\X_1,\]
and
\begin{multline*}
	P_\sigma\Big(Y=\frac{1}{\sqrt{\be}}\, | \, X=x \Bigr) = 
\frac{1+\sigma\sqrt{\be}}{2} \\ = 1 - P_\sigma \Bigl( Y=-\frac{1}{\sqrt{\be}}
\, | \, X=x \Bigr) 
		\quad \text{for any } x\in\X_2.
\end{multline*}
One can easily check that for any $\sigma\in\{-,+\}$, $\Var_{P_\sigma}(Y) =1-\be \le 1$ and
$\freg(x) = \sigma \vp_2 \in\cF$.
To prove Theorem \ref{th:lb}, it suffices to prove \eqref{eq:lb} when the supremum is taken 
among $P\in\{P_-,P_+\}$. This is done by applying Theorem 8.2 of \cite{Aud08}.
Indeed, the pair $(P_-,P_+)$ forms a $(1,\be,\be)$-hypercube in the sense of Definition 8.2
with edge discrepancy of type I (see (8.5), (8.11) and (10.20) for $q=2$): 
$d_I = 1$. We obtain
	\[
	\und{\sup}{P\in\{P_-,P_+\}} \big\{ \E \bigl[ R(\hf) \bigr] - R(f^*) \big\}
		\ge \be(1-\be\sqrt{n}),
	\]
which gives the desired result by taking $\be=1/(2\sqrt{n}).$
\end{proof}

\section{Empirical risk minimization on a ball: analysis derived from the work of Birg\'e and Massart} \label{sec:bm}

We will use the following covering number upper bound 
\cite[Lemma~1]{Lor1966}

\begin{lemma} \label{le:infty}
If $\cF$ has a diameter upper bounded by $H$ for the $L^\infty$-norm (i.e., $\sup_{f_1,f_2\in \cF,x\in\X} |f_1(x)-f_2(x)| \le H$), then 
for any $0<\delta\le H$, there exists a set $\Fg \subset \cF$, of cardinality $|\Fg|\le (3H/\delta)^d$ such that
for any $f\in\cF$ there exists $g\in\Fg$ such that $\|f-g\|_{\infty} \le \delta.$
\end{lemma}

We apply a slightly improved version of Theorem~5 in Birg\'e and Massart \cite{BirMas98}.
First for homogeneity purpose, we modify Assumption M2 by replacing the 
condition ``$\sigma^2 \ge D/n$'' by ``$\sigma^2 \ge B^2 D/n$'' where
the constant $B$ is the one appearing in (5.3) of \cite{BirMas98}. 
This modifies Theorem 5 of \cite{BirMas98} to the extent that ``$\vee 1$'' 
should be replaced with ``$\vee B^2$''.
Our second modification is to remove the assumption that $W_i$ and $X_i$ 
are independent. A careful look at the proof shows that the result
still holds when (5.2) is replaced by: for any $x\in\X$, and $m\ge2$
	\[\text{E}_s [M^m(W_i)|X_i=x] \le a_m A^m, \qquad\text{for all } i=1,\dots,n.\]
We consider 
	$W = Y - f^*(X)$,
	$\gamma(z,f) = (y-f(x))^2$,
	$\Delta(x,u,v) = |u(x)-v(x)|$,
and
	$M(w)= 2 ( |w| + H )$.
From \eqref{eq:expmom}, for all $m\ge 2$, we have
	$\E \big\{ [(2 (|W|+H)]^m|X=x] \le \frac{m!}{2} [4M (A+H)]^m.$
Now consider $B'$ and $r$ such that Assumption M2 of \cite{BirMas98} holds for $D=d$.
Inequality (5.8) for $\tau=1/2$ of \cite{BirMas98} implies that for any $v\ge \kap \frac{d}{n} (A^2+H^2) \log(2B'+B'r \sqrt{d/n})$,
with probability at least $\ds 1 - \kap \exp \Bigl[ \frac{-n v}{\kap(A^2+H^2)} 
\Bigr]$,
	\[
	R(\hferm)-R(f^*)+r(f^*)-r(\hferm) \le \bigl( \E
\bigl\{ \bigl[\hferm(X)-f^*(X)\bigr]^2 \bigr\} \vee v \bigr)/2  
	\]
for some large enough constant $\kap$ depending on $M$.
Now from Proposition 1 of \cite{BirMas98} and Lemma \ref{le:infty}, one can take 
either $B'=6$ and $r \sqrt{d} = \sqrt{\cR}$ or $B'=3\sqrt{n/d}$ and
$r=1$.  
By using $\E\bigl\{ \bigl[\hferm(X)-f^*(X) \bigr]^2 
\bigr\} \le R(\hferm)-R(f^*)$ (since $\cF$ is convex
and $f^*$ is the orthogonal projection of $Y$ on $\cF$), and
$r(f^*)-r(\hferm)\ge 0$ (by definition of $\hferm$), the desired result
can be derived.

Theorem \ref{th:bmnew} provides a $d/n$ rate provided that the geometrical quantity $\cR$ is at most of order $n$.
Inequality (3.2) of \cite{BirMas98} allows to bracket $\cR$ in terms
of $\bcR = \sup_{f\in\Span \{\vp_1,\dots,\vp_d\}} \fracc{\|f\|_\infty^2}{\E[f(X)]^2}$, namely
	$\bcR \le \cR \le \bcR d$.
To understand better how this quantity behaves and to illustrate some of the presented
results, let us give the following simple example.

{\bf Example 1.} \label{ex1a}
Let $A_1,\dots,A_d$ be a partition of $\X$, i.e., $\X =\sqcup_{j=1}^d A_j$.
Now consider the indicator functions $\vp_j=\dsone_{A_j}, j=1,\dots,d$: $\vp_j$ is equal
to $1$ on $A_j$ and zero elsewhere. Consider 
that $X$ and $Y$ are independent and that $Y$ is a Gaussian random variable with mean $\theta$ and 
variance $\sigma^2$. In this situation: $\flin=\freg=\sum_{j=1}^d \theta \vp_j$.
According to Theorem \ref{th:weakols}, if we know an upper bound $H$ on $\|\freg\|_\infty=\th$, we have that the
truncated estimator $(\hfols\wedge H)\vee -H$ satisfies 
	\[
	\E R(\hfols_H) - R(\flin)	\le \kap \frac{(\sigma^2\vee H^2)d\log n }{n}
	\]
for some numerical constant $\kap$.
Let us now apply Theorem \ref{th:capvit}.
Introduce $p_j=\P(X\in A_j)$ and $p_{\min} = \min_{j} p_j$. 
We have $Q = \big( \E \vp_j(X) \vp_k(X) \big)_{j,k} = \diag(p_j)$,
$\cK=1$ and $\|\th^*\|=\th\sqrt{d}$. We can take $A=\sigma$ and $M=2$.
From Theorem \ref{th:capvit}, for $\lam=d \cL_\eps/n$, as soon as $\lam \le p_{\min}$,
the ridge regression estimator satisfies with probability at least $1-\eps$:
	\beglab{eq:ex1cv}
	R(\hfrlam) - R( \flin ) \le 
		\kap \cL_\eps \frac{d}{n} \bigg( \sigma^2 + \frac{\th^2 d^2 \cL^2_\eps}{n p_{\min}} \bigg)
	\endlab
for some numerical constant $\kap$. When $d$ is large, the term $\fracb{d^2 \cL^2_\eps}{n p_{\min}}$ is felt, 
and leads to suboptimal rates. Specifically, since $p_{\min}\le 1/d$,
the $\rhs$ of \eqref{eq:ex1cv} is greater than $d^4/n^2$, which is much larger than $d/n$ when $d$ is much larger than $n^{1/3}$. 
If $Y$ is not Gaussian but almost surely uniformly bounded by $C<+\infty$, then the randomized estimator proposed in Theorem \ref{th:alqa}
satisfies the nicer property: with probability at least $1-\eps$,
    \[
    R(\hat{f}) - R(\flin) \le \kap (H^2 +C^2) \frac{d \log( 3p_{\min}^{-1} ) + 
        \log ( (\log n) \eps^{-1} ) }{n},
    \]
for some numerical constant $\kap$.
In this example, one can check that $\cR=\cR'=1/p_{\min}$ where $p_{\min}= \min_j \P(X\in A_j).$
As long as $p_{\min}\ge 1/n$, the target \eqref{eq:devtarget} is reached from Corollary \ref{th:bmnew}.
Otherwise, without this assumption, the rate is in $(d\log(n/d))/n$. 
$\blacksquare$ 

\section{Ridge regression analysis from the work of Caponnetto and De Vito} \label{sec:capvit}

From \cite{CapVit07}, one can derive the following risk bound for the ridge estimator.

\begin{thm} \label{th:capvit}
Let $q_{\min}$ be the smallest eigenvalue of the $d \times d$-product matrix $Q=\big( \E \vp_j(X) \vp_k(X) \big)_{j,k}$.
Let $\cK= \sup_{x\in\X} \sum_{j=1}^d \vp_j(x)^2$. Let $\|\th^*\|$ be the Euclidean norm of 
the vector of parameters of $\flin=\sum_{j=1}^d \th^*_j \vp_j$.
Let $0<\eps<1/2$ and $\cL_\eps = \leps$. 
Assume that for any $x\in\X,$ 
	\[\E \Bigl\{ \exp \bigl[ |Y-\flin(X)|/A \bigr]\, | \, X= x 
\Bigr\} \le M.
	\]
For $\lam=\fracl{\cK d \cL_\eps}{n}$, if
$\lam \le q_{\min}$,
the ridge regression estimator satisfies with probability at least $1-\eps$:
	\beglab{eq:capvit}
	R(\hfrlam) - R( \flin ) \le 
		\frac{\kap \cL_\eps d}{n} \bigg( A^2 + \frac{\lam}{q_{\min}} \cK \cL_\eps \|\theta^*\|^2 \bigg)
	\endlab
for some positive constant $\kap$ depending only on $M$.
\end{thm}

\begin{proof} 
One can check that 
	$
	\hfrlam \in \undc{\argmin}{f\in\cH} 
	            r(f) + \lam \sum_{j=1}^d \|f\|_{\cH}^2,
	$
where $\cH$ is the reproducing kernel Hilbert space associated with the kernel $K:(x,x') \mapsto \sum_{j=1}^d \vp_j(x) \vp_k(x')$.
Introduce 
	$
	f^{(\lam)} \in \undc{\argmin}{f\in\cH} R(f) + \lam \sum_{j=1}^d \|f\|_{\cH}^2.
	$
Let us use Theorem 4 in \cite{CapVit07} and the notation defined in their Section 5.2.
Let $\vp$ be the column vector of functions $[\vp_j]_{j=1}^d$, 
$\diag(a_j)$ denote the diagonal $d \times d$-matrix whose $j$-th element on the diagonal is $a_j$,
and $I_d$ be the $d \times d$-identity matrix.
Let $U$ and $q_1,\dots,q_d$ be such that $UU^T=I$ and 
	$Q=U \diag(q_j)U^T$.
We have $\flin = \vp^T \th^*$
and $f^{(\lam)} = \vp^T (Q+\lam I)^{-1} Q \th^*$, hence
	$$
	\flin - f^{(\lam)} = \vp^T U \diag(\lam/(q_j+\lam)) U^T \th^*.
	$$
After some computations, we obtain that the residual, reconstruction error and effective dimension respectively satisfy
$\A(\lam) \le \frac{\lam^2}{q_{\min}} \|\th^*\|^2$, $\cB(\lam) \le \frac{\lam^2}{q_{\min}^2} \|\th^*\|^2$,
and $\N(\lam) \le d$. The result is obtained by noticing that the leading terms in (34) of \cite{CapVit07} are
$\A(\lam)$ and the term with the effective dimension $\N(\lam)$. 
\end{proof}

The dependence in the sample size $n$ is correct since $1/n$ is known to be minimax optimal. 
The dependence on the dimension $d$ is not optimal, as it is observed in the example given page \pageref{ex1a}. 
Besides the high probability bound \eqref{eq:capvit}
holds only for a regularization parameter $\lam$ depending on the confidence level $\eps$. So we do not have a single estimator 
satisfying a PAC bound for every confidence level. 
Finally the dependence on the confidence level is larger than expected. It contains an unusual square.
The example given page \pageref{ex1a} illustrates Theorem \ref{th:capvit}.

\section{Some standard upper bounds on log-Laplace trans\-forms}

\begin{lemma} \label{le:stda}
Let $V$ be a random variable almost surely bounded by $b\in\R$.
Let $g: u\mapsto \bigl[ \exp(u) - 1 - u \bigr]/u^2$.
    \[
    \log \Bigl\{  \E  \Bigr[ \exp \bigl[ V-\E (V) \bigr] 
\Bigr] \Bigr\} \le \E \bigl( V^2 \bigr)  g(b).
    \]
\end{lemma}

\begin{proof}
Since $g$ is an increasing function, we have $g(V) \le g(b)$. By using the inequality
$\log(1+u) \le u$, we obtain
\begin{multline*}
    \log 
\Bigl\{ \E  \Bigl[ \exp \bigl[ V-\E (V) \bigr] \Bigr] \Bigr\} = 
-\E (V) + \log \bigl\{ \E \bigl[ 1+V+V^2g(V)
\bigr] \bigr\} \\ \le \E \bigl[ V^2 g(V)\bigr]  \le \E 
\bigl( V^2 \bigr) g(b).
\end{multline*}
\end{proof}

\begin{lemma} \label{le:stdb}
Let $V$ be a real-valued random variable such that 
$\E \bigl[ \exp \bigl( |V| \bigr) \bigr] \le M$ for some $M>0$.
Then we have $|\E (V)| \le \log M$, and for any $-1<\alpha<1$,
    \[
    \log 
\Bigl\{ \E  \Bigr[ \exp \bigl\{  \alpha  \bigl[ V-\E (V) 
\bigr] \bigr\} \Bigr] \Bigr\} \le \frac{\alpha^2 M^2}{2\sqrt{\pi}(1-|\alpha|)}.
    \]
\end{lemma}

\begin{proof}
First note that by Jensen's inequality, we have $|\E (V)|\le \log(M)$.
By using $\log(u) \le u-1$ and Stirling's formula, for any $-1<\alpha<1$, we have
    \begin{multline*}
    \log 
\Bigl\{ \E \Bigl[ \exp \bigl\{ \alpha  \bigl[ V-\E (V) 
\bigr] \bigr\} \Bigr] \Bigr\} \le \E 
\Bigl[ \exp \bigl\{ \alpha  \bigl[ V-\E (V) \bigr] 
\bigr\} \Bigr] \Bigr\}  - 1\\
     = \E \Bigl\{ \exp \bigl\{ \alpha \bigl[ V-\E (V) 
\bigr] \bigr\} - 1 - \alpha \bigl[ V-\E (V) \bigr] 
\Bigr\} \\
     \le \E 
\Bigl\{ \exp \bigl[ |\alpha| |V-\E (V)| \bigr] - 1 - |\alpha| |V-\E (V)| 
\Bigr\} \\
     \le \E \Bigl\{ \exp \bigl[ |V-\E (V)| 
\bigr] \Bigr\} \sup_{u\ge 0} 
\Bigl\{ \bigl[ \exp (|\alpha| u) - 1 - |\alpha| u 
\bigr]\exp(-u) \Bigr\}\\
     \le \E \Bigl[ \exp \bigl(|V|+|\E (V)| \bigr) \Bigr]  
\sup_{u\ge 0} \sum_{m\ge 2} \frac{|\alpha|^m u^m}{m!} \exp(-u)\\
     \le M^2 \sum_{m\ge 2} \frac{|\alpha|^m}{m!} \sup_{u\ge 0} u^m \exp(-u)
     = \alpha^2 M^2 \sum_{m\ge 2} \frac{|\alpha|^{m-2}}{m!} m^m \exp(-m) \\
     \le \alpha^2 M^2 \sum_{m\ge 2} \frac{|\alpha|^{m-2}}{\sqrt{2\pi m}}
     \le \frac{\alpha^2 M^2}{2\sqrt{\pi}(1-|\alpha|)}.
    \end{multline*}
\end{proof}

\bibliographystyle{plain}
\bibliography{ref}

\end{document}